\documentclass[a4paper]{article}
\usepackage[all]{xy}\usepackage[latin1]{inputenc}        
\usepackage[dvips]{graphics,graphicx}
\usepackage{amsfonts,amssymb,amsmath,color,mathrsfs, amstext}
\usepackage{amsbsy, amsopn, amscd, amsxtra, amsthm,authblk}
\usepackage{enumerate,algorithmicx,algorithm}
\usepackage{algpseudocode}
\usepackage{upref}
\usepackage{geometry}
\usepackage[displaymath]{lineno}
\usepackage{float}
\usepackage{yhmath}
\usepackage{booktabs}
\usepackage{subcaption}
\usepackage{multirow}
\usepackage{makecell}
\usepackage[colorlinks,
            linkcolor=red,
            anchorcolor=red,
            citecolor=red
            ]{hyperref}

\numberwithin{equation}{section}

\def\e{\epsilon}

\def\sq{\mathsf{q}}

\newcommand{\E}{\mathbb{E}}
\newcommand{\R}{\mathbb{R}}
\newcommand{\bbP}{\mathbb{P}}
\newcommand{\bfP}{\mathbf{P}}

\newcommand{\cL}{\mathcal{L}}

\newtheorem{theorem}{Theorem}[section]
\newtheorem{lemma}{Lemma}[section]
\newtheorem{definition}{Definition}[section]
\newtheorem{remark}{Remark}[section]
\newtheorem{proposition}{Proposition}[section]
\newtheorem{assumption}{Assumption}[section]

\numberwithin{equation}{section}

\begin{document}

\title{On the mean field limit of the Random Batch Method for interacting particle systems \footnote{Accepted for publication in SCIENCE CHINA Mathematics.}}
\author[1]{Shi Jin \thanks{shijin-m@sjtu.edu.cn}}
\author[2]{Lei Li\thanks{leili2010@sjtu.edu.cn}}
\affil[1,2]{School of Mathematical Sciences, Institute of Natural Sciences, MOE-LSC, Shanghai Jiao Tong University, Shanghai, 200240, P. R. China.}

\date{}
\maketitle

\begin{abstract}
The Random Batch Method proposed in our previous work [Jin et al., J. Comput. Phys., 400(1), 2020]  is not only a numerical method for interacting particle systems and its mean-field limit, but also can be viewed as a model of particle system in which particles interact, at discrete time,  with randomly selected mini-batch of particles. In this paper we investigate the mean-field limit of this model as the number of particles $N \to \infty$.   Unlike the classical mean field limit for interacting particle systems where the law of large numbers plays the role and the chaos is propagated to later times, the mean field limit now does not rely on the law of large numbers and chaos is imposed at every discrete time.  Despite this,  we will not only justify this mean-field limit (discrete in time)  but will also show that  the limit, as the discrete time interval $\tau \to 0$, approaches to the solution of a nonlinear Fokker-Planck equation arising as the mean-field limit of the original interacting particle system  in Wasserstein distance. 
\end{abstract}

\section{Introduction}

Many physical, biological and social sciences phenomena, at the microscopic level, are described by interacting particle systems, for example, molecules in fluids \cite{frenkel2001understanding}, plasma \cite{birdsall2004}, swarming \cite{vicsek1995novel,carrillo2017particle,carlen2013kinetic,degond2017coagulation}, chemotaxis \cite{horstmann03,bertozzi12}, flocking  \cite{cucker2007emergent,hasimple2009,albi2013}, synchronization \cite{choi2011complete,ha2014complete}
 and consensus \cite{motsch2014}.
We consider  the following general first order systems
\begin{gather}\label{eq:Nparticlesys}
dX^i=b(X^i)\,dt+\frac{1}{N-1}\sum_{j: j\neq i} K(X^i-X^j)\,dt+\sqrt{2}\sigma\, dW^i,~~i=1,2,\cdots, N,
\end{gather}
with the initial data $X^i_0$'s being independent and identically distributed (i.i.d.), sampled from a common distribution
$\mu_0$. $W^i$'s are $N$ independent $d$-dimensional Wiener processes (standard Brownian motions). Here, we allow $\sigma=0$ to include  systems without noise. 

As is well-known, under certain conditions, the mean field limit (i.e., $N\to\infty$) of \eqref{eq:Nparticlesys} is given by 
\begin{gather}\label{eq:nonlinearFP}
\partial_t\mu=-\nabla\cdot((b(x)+K*\mu)\mu)+\sigma^2\Delta\mu.
\end{gather}
This means that the empirical measure $\mu_N:=N^{-1}\sum_{i=1}^N \delta(x-X^i)$ converges weakly to $\mu$ almost surely and the one marginal distribution $\mu_N^{(1)}:=\mathscr{L}(X^1)$, the law of $X^1$, converges to $\mu$.
See  \cite{cattiaux2008,fournier2014,durmus2018elementary,li2019} for some related models and proofs, though the setups in these works do not quite fit our problem as we allow $|b(\cdot)|$ to have polynomial growth.
Recall that $\mu$ is in general a probability distribution and \eqref{eq:nonlinearFP} is understood in the distributional sense.  We will denote the solution operator to \eqref{eq:nonlinearFP} by $\mathcal{S}$:
\begin{gather}
\mathcal{S}(\Delta)\mu(t_1):=\mu(t_1+\Delta),~\forall t_1\ge 0, \Delta\ge 0.
\end{gather}
Clearly, $\{\mathcal{S}(t): t\ge 0\}$ is a nonlinear semigroup.

Direct simulation of \eqref{eq:Nparticlesys} costs $O(N^2)$ per time step, which is expensive. To reduce the computational cost, in \cite{jin2020random},  a random algorithm that uses random mini-batches, called the Random Batch Method (RBM), has been proposed to reduce the computation cost per time step from $O(N^2)$ to $O(N)$. The method has been applied to various problems with promising results \cite{jin2020random,li2019stochastic,ko2020model,li2020direct}. However, the understanding of the method is still limited, despite some theoretical proofs \cite{jin2020random,jin2020convergence}.
 The idea of using the ``mini-batch" was inspired by the stochastic gradient descent (SGD) method \cite{robbins1951,bottou1998online} in machine learning. The ``mini-batch" was also used for Bayesian inference  \cite{welling2011bayesian}, and similar ideas were used to simulate the mean-field equations for flocking \cite{albi2013}.  How to apply the mini-batch depends on the specific problems. The strategy in \cite{jin2020random} for interacting particle systems \eqref{eq:Nparticlesys} is to do random grouping. 
Intuitively, the method converges due to certain time average in time, and thus the convergence is like the convergence in the Law of Large Number (in time). See \cite{jin2020random}  for more details. Compared with the Fast Multipole Method, the accuracy is lower (half order in time step), but RBM is simpler to implement and is valid for more general potentials (\cite{li2019stochastic,jin2020convergence}). 

The RBM algorithm corresponding to \eqref{eq:Nparticlesys} is shown in Algorithm \ref{rbm}. Suppose we aim to do simulation until time $T>0$. We first choose a time step $\tau>0$ and a batch size $p\ll N, p\ge 2$ that divides $N$. Define the discrete time grids $t_k:=k\tau$, $k\in \mathbb{N}$. For each time subinterval $[t_{k-1}, t_k)$, there are two steps: (1) at time grid $t_{k-1}$, we divide the $N$ particles into $n:=N/p$ groups (batches) randomly; (2) the particles evolve with interaction inside the batches only. Here, we use the same symbols $X^i$ without causing any confusion. The Wiener process $W^i$ (Brownian motion) used in \eqref{eq:rbmSDE} is the same as in \eqref{eq:Nparticlesys}. 
\begin{algorithm}[H]
\caption{(RBM)}\label{rbm}
\begin{algorithmic}[1]
\For {$k \text{ in } 1: [T/\tau]$}   
\State Divide $\{1, 2, \ldots, N\}$ into $n=N/p$ batches randomly.
     \For {each batch  $\mathcal{C}_q$} 
     \State Update $X^i$'s ($i\in \mathcal{C}_q$) by solving the following SDE with $t\in [t_{k-1}, t_k)$.
     \begin{gather}\label{eq:rbmSDE}
            d X^i=b(X^i) dt+\frac{1}{p-1}\sum_{j\in \mathcal{C}_q,j\neq i} K(X^i-X^j)dt+\sqrt{2}\sigma\, dW^i.
      \end{gather}
      \EndFor
\EndFor
\end{algorithmic}
\end{algorithm}

As pointed out in \cite{jin2020random}, RBM is asymptotic-preserving regarding the mean field limit $N\to\infty$ (\cite{stanley1971, georges1996, lasry2007}); namely, the error bound of the one marginal distribution can be made independent of $N$ so that it can be used for large $N$ as an efficient numerical particle method for \eqref{eq:nonlinearFP}, the mean field nonlinear Fokker-Planck equation of \eqref{eq:Nparticlesys}. 
While RBM was introduced as a numerical method, it can also be viewed as a new model for the underlying particle system. A natural question, for both numerical and modeling interests, is: what is the limiting (mean field) dynamics as $N\to\infty$ for a fixed time step $\tau$?

Intuitively, in a specific realization of the random division of batches, when $N\gg 1$, the probability that two chosen particles are correlated is very small. Hence, in the $N\to\infty$ limit, the two chosen particles will be uncorrelated with probability $1$. Since the particles are exchangeable, the marginal distributions of them will be identical. Hence, let us focus on one specific particle, say $i=1$, to understand the mean field limit. Imagine that there are infinitely many particles as $N\to\infty$. For each time interval, we draw $p-1$ particles from the infinite set, and they are independent from particle $1$ by the intuition just mentioned. They share the same distribution with particle $1$. This small group then evolves with interactions between themselves to the next time point so that the distribution of particle $1$ has been changed.  At this new time point, we draw another $p-1$ particles to interact with particle $1$. In this sense, in the $N\to\infty$ limit, the $N$-particle system is then reduced to a $p$-particle system described by the following SDE system for $t\in [t_k, t_{k+1})$:
\begin{gather}\label{eq:RBMmeanfieldSDE}
dY^i=b(Y^i)\,dt+\frac{1}{p-1}\sum_{j=1,j\neq i}^{p}K(Y^i-Y^j)\,dt
+\sqrt{2}\sigma\,dW^i,~~i=1,\cdots, p,
\end{gather}
with $\{Y^i(t_k)\}$ being i.i.d., drawn from $\tilde{\mu}(\cdot, t_k)$.  
We may impose $Y^1(t_k^-)=Y^1(t_k^+)$, and for other particles $i\neq 1$,  $Y^i(t)$ in $[t_{k-1}, t_k)$ and $[t_k, t_{k+1})$ are independent so they are not continuous at $t_k$. In fact, $Y^i$'s $(i\neq 1)$ correspond to the batchmates of particle $1$ as in Algorithm \ref{rbm} so they are different particles for different iterations. Then, $\tilde{\mu}(\cdot, t_{k+1})=\mathscr{L}(Y^1(t_{k+1}^-))$, the law of $Y^1(t_{k+1}^-)$. In terms the individual particle $1$, the rest $N-1$ particles average out to an infinite pool of independent particles from particle $1$ at each time step $t_k$. This becomes the mean field limit model of RBM, and one may write out the following mean field limit for RBM in terms of the probability distribution as shown in Algorithm \ref{meanfield}, while \eqref{eq:RBMmeanfieldSDE} becomes the microscopic description.

\begin{algorithm}[H]
\caption{(Mean Field Dynamics of RBM \eqref{eq:rbmSDE})}\label{meanfield}
\begin{algorithmic}[1]
\State $\tilde{\mu}(\cdot, 0)=\mu_0$.
\For {$k \ge 0$}  

\State Let $\rho^{(p)}(\cdots, t_k)=\tilde{\mu}(\cdot, t_{k})^{\otimes p}$ be a probability measure on $(\R^{d})^{p}\cong \R^{pd}$.

\State Evolve the measure $\rho^{(p)}$ by the following Fokker-Planck equation for $t\in [t_k, t_{k+1})$:
\begin{gather}\label{eq:firstalgorithm}
            \partial_t\rho^{(p)}=-\sum_{i=1}^p 
            \nabla_{x_i}\cdot\left(\Big[b(x_i)+\frac{1}{p-1}\sum_{j=1,j\neq i}^p K(x_i-x_j)\Big]\rho^{(p)}\right)+\sigma^2\sum_{i=1}^p \Delta_{x_i}\rho^{(p)}.
\end{gather}

\State Set
\begin{gather}
\tilde{\mu}(\cdot, t_{k+1}):=\int_{(\R^{d})^{(p-1)}}
\rho^{(p)}(\cdot,dy_2,\cdots,dy_p, t_{k+1}^-).
\end{gather}

\EndFor
\end{algorithmic}
\end{algorithm}
The dynamics shown in Algorithm \ref{meanfield} naturally defines a nonlinear operator $\mathcal{G}_{\infty}: \bfP(\R^d)\to \bfP(\R^d)$ as
\begin{gather}\label{eq:Ginfty}
\tilde{\mu}(\cdot, t_{k+1})=: \mathcal{G}_{\infty}(\tilde{\mu}(\cdot, t_k)).
\end{gather}

As indicated above, the mean field limit here does not rely on the law of large numbers. Instead, it relies on the fact that the particles in one batch are unlikely to be related if $N\gg 1$. In the mean field limit dynamics of RBM, one starts with a chaotic configuration \footnote{By ``chaotic configuration", we mean that there exists a one particle distribution $f$ such that for any $j$, the $j$-marginal distribution is given by $\mu^{(j)}=f^{\otimes j}$. Such independence in a configuration is then loosely called ``chaos". If the $j$-marginal distribution is more close to $f^{\otimes j}$ for some $f$, we loosely say ``there is more chaos".}, the $p$ particles evolve with interaction to each other. Then, at the starting point of the next time interval, one imposes the chaos so that the particles are independent again.  This mean field limit is different from the standard mean field limit for system \eqref{eq:Nparticlesys}, given by \eqref{eq:nonlinearFP}: in the mean field limit of RBM, the chaos is imposed at every time step;  in the classical mean field limit for interacting particle system, the chaos is propagated to later times. 
This mechanism may allow the mean-field limit of RBM to achieve a higher convergence rate than the standard $N^{-1/2}$ convergence rate (at least $N^{-1}$ under Wasserstein-$1$ as seen in section \ref{sec:pfmeanfield}).  In spite of the difference just mentioned, we will show that  these two limiting dynamics are in fact close: in section \ref{sec:tautozero}, we will show that as $\tau\to 0$ the dynamics given by $\mathcal{G}_{\infty}$ can approximate that of the nonlinear Fokker-Planck equation \eqref{eq:nonlinearFP}. We remark that as $\tau\to 0$, the dynamics of RBM has been shown to converge to the $N$-particle system \eqref{eq:Nparticlesys} in \cite{jin2020random}.  Thus, this result implies that the two limits $\lim_{N\to \infty}$ and $\lim_{\tau\to 0}$ commute (see also section \ref{subsec:limitcommute} and Fig. \ref{fig:sde}).
 
The argument in this paper for $t\le T$ can be generalized to second order systems, which we omit, but one may see section \ref{sec:discussion} for some discussion.
 Of course, the argument for large time behavior
can be different and this is left for future study.

The rest of the paper is organized as follows. We introduce the notations and give a brief review to Wasserstein distance in section \ref{sec:setup}. The mean field limit under Wasserstein distance is shown in section \ref{sec:pfmeanfield}.
Section \ref{sec:tautozero} is devoted to the discussion of the mean field dynamics of RBM. In particular, we show that it is close to the mean-field nonlinear Fokker-Planck equation. Some discussion is performed in section \ref{sec:discussion}.  We finally conclude the work along with future directions  in section \ref{sec:conclusion}.

\section{Preliminaries and notations }\label{sec:setup}

In this section, we first introduce some assumptions and notations. Then we give a brief introduction to Wasserstein distances and prove some auxiliary results.

\subsection{Mathematical setup of the problem}

We first introduce several assumptions that will be used throughout the paper.
In these assumptions, ``being smooth" means that the functions are infinitely differentiable. Note that the conditions in these assumptions may be stronger than necessary. 
\begin{assumption}\label{ass:momentmu0}
The moments of the initial data are finite:
\begin{gather}
\int_{\R^d} |x|^q\mu_0(dx)<\infty,~\forall q\in [1,\infty).
\end{gather}
\end{assumption}
One of the following two conditions will be used for the external fields and interaction kernels. 
\begin{assumption}\label{ass:kernelfunctions}
Assume $b(\cdot): \R^d\to\R^d$ and $K(\cdot): \R^d\to \R^d$ are smooth.
Moreover,  $b(\cdot)$ is one-sided Lipschitz:
\begin{gather}
(z_1-z_2)\cdot (b(z_1)-b(z_2))\le \beta |z_1-z_2|^2
\end{gather}
for some constant $\beta$, and $K$ is Lipschitz continuous
\[
|K(z_1)-K(z_2)|\le L|z_1-z_2|.
\]
\end{assumption}

\begin{assumption}\label{ass:kernelstrong}
The fields $b(\cdot): \R^d\to\R^d$ and $K(\cdot): \R^d\to \R^d$ are smooth. Moreover,  $b(\cdot)$ is strongly confining:
\begin{gather}
(z_1-z_2)\cdot (b(z_1)-b(z_2))\le -r |z_1-z_2|^2
\end{gather}
for some constant $r>0$, and $K$ is Lipschitz continuous
$|K(z_1)-K(z_2)|\le L|z_1-z_2|$.
The parameters $r, L$ satisfy
\begin{gather}
r>2L.
\end{gather}
\end{assumption}

\begin{remark}
Compared with our previous works \cite{jin2020random,jin2020convergence}, we are not assuming the boundedness of $K$
in this paper to prove the mean-field limit and investigate the limiting dynamics. The boundedness of $K$ in our previous works is a simple condition to guarantee the boundedness of the variance of the random forces (though the boundedness of variance may also be proved without assuming boundedness of $K$). 
\end{remark}

Denote $\mathcal{C}_q^{(k)}$ ($1\le q\le n$) the batches at $t_k$ so that $\cup_{q}\mathcal{C}_q^{(k)}=\{1,\cdots, N\}$, and
\begin{gather}
\mathcal{C}^{(k)}:=\{\mathcal{C}_1^{(k)}, \cdots, \mathcal{C}_n^{(k)}\}
\end{gather}
will denote the random division of batches at $t_{k}$.
By the Kolmogorov extension theorem \cite{durrett2010}, there exists a probability space $(\Omega, \mathcal{F}, \bbP)$ such that the random variables $\{X_0^{i}, W^i, \mathcal{C}^{(k)}: 1\le i\le N, k\ge 0\}$ are defined on this probability space and are all independent.
We will use $\E$ to denote the integration on $\Omega$ with respect the probability measure $\bbP$.
For the convenience of the analysis, we introduce the $L^2(\bbP)$
norm as:
\begin{gather}
\|v\|:=\sqrt{\mathbb{E}|v|^2}.
\end{gather}
Define the filtration $\{\mathcal{F}_{k}\}_{k\ge 0}$ by
\begin{gather}
\mathcal{F}_{k-1}=\sigma(X_0^i, W^i(t), \mathcal{C}^{(j)}; t\le t_{k-1}, j\le k-1).
\end{gather}
Clearly, $\mathcal{F}_{k-1}$ is the $\sigma$-algebra generated by the initial values $X_0^i$ ($i=1,\ldots, N$), $W^i(t)$, $t\le t_{k-1}$, and $\mathcal{C}^{(j)}$, $j\le k-1$. Hence, $\mathcal{F}_{k-1}$ contains the information of how batches are constructed for $t\in [t_{k-1}, t_k)$.

\subsection{A review of the Wasserstein distance}

Consider a domain $O\subset \mathbb{R}^n$ where $n$ is a positive integer. We denote $\bfP(O)$ the set of probability measures on $O$. Let $\mu,\nu\in \bfP(O)$ be two probability measures and $c: O\times O \to [0,\infty)$ be a cost function. One solves the following optimization problem for the optimal transport:
\[
\min_{\gamma}\left\{\int_{O\times O} c \,d\gamma \Big| \gamma\in \Pi(\mu,\nu) \right\} ,
\]
where $\Pi(\mu, \nu)$ is the set of ``transport plans'', i.e., a joint measure on $O\times O$ such that the marginal measures are $\mu$ and $\nu$ respectively. If there is a map $T: O\to O$ such that $(I\times T)_{\#}\mu$ minimizes the target function,  then $T$ is called an optimal transport map. Here, $I$ is the identity map and
\begin{gather}\label{eq:pushforward}
(I\times T)_{\#}\mu (E):=\mu((I\times T)^{-1}(E)),~\forall E\subset O\times O,~\mathrm{measurable}.
\end{gather}

Choosing the particular cost function $c(x,y)=|x-y|^\sq$, $\sq\in [1,\infty)$, one can define the Wasserstein-$\sq$ distance $W_{\sq}(\mu, \nu)$ as
\begin{gather}\label{eq:Wassersteindef}
W_\sq(\mu, \nu):=\left(\inf_{\gamma\in \Pi(\mu,\nu)}\int_{O\times O} |x-y|^{\sq} d\gamma\right)^{1/\sq}.
\end{gather}

It has been shown (see \cite{Benamou2000}, \cite[Chap. 5]{santambrogio15}) that the Wasserstein-$\sq$ distance between two probability measures $\mu$ and $\nu$ is also given by
\begin{gather}\label{eq:Wp}
W_{\sq}^{\sq}(\mu, \nu)=\min\left\{\int_0^1\|v\|^{\sq}_{L^{\sq}(\rho)}dt: \partial_t\rho+\nabla\cdot(\rho v)=0, \rho|_{t=0}=\mu, \rho|_{t=1}=\nu \right\} ,
\end{gather}
where $\rho$ is a (time-parametrized) nonnegative measure and
\begin{gather}
\|v\|^{\sq}_{L^{\sq}(\rho)}:=\int_{O}|v|^{\sq} \rho(dx).
\end{gather}
 Hence, $v$ can be thought as the particle velocity for the optimal transport, as explained in \cite[Chap. 5]{santambrogio15}.
With this explanation, one can then understand $\bfP(O)$ equipped with $W_2$ distance as a Riemannian manifold so that the Fokker-Planck equations can be formulated as a class of gradient flows on this manifold (see, for example, \cite{jordan1998}, \cite[Chap. 8]{villani2003topics}).

Below, we note a useful lemma that relates the total variation distance to the $W_{\sq}$ distance. This is intrinsically \cite[Proposition 7.10]{villani2003topics} and the version here is more convenient for our purpose in this paper. Recall the Jordan decomposition for a signed measure $\mu=\mu^+-\mu^-$ defined on a Polish space $\mathcal{E}$. Then, define $|\mu|:=\mu^++\mu^-$, and the total variation norm of the signed measure by
\begin{gather}
\|\mu\|_{TV}:=|\mu|(\mathcal{E})=\mu^{+}(\mathcal{E})+\mu^{-}(\mathcal{E}).
\end{gather}
\begin{lemma}\label{lmm:W2byTV}
Let $\mu, \nu\in \bfP(\R^d)$ be two different probability measures on $\mathcal{E}=\R^d$. Let $\delta\ge 0$ and $\hat{\mu}$ be a measure such that
$|\mu-\nu|(E)\le \delta\, \hat{\mu}(E)$, for any Borel measurable $E$.
Suppose for $\sq\ge 1$, $M_{\sq}:=\inf_{x_0}\int_{\R^d} |x-x_0|^{\sq} \hat{\mu}(dx)<\infty$.
Then,
\begin{gather}
W_{\sq}(\mu,\nu)\le 2^{1-1/\sq}(M_{\sq}\delta)^{1/\sq}.
\end{gather} 

In particular, choosing $\delta=\|\mu-\nu\|_{TV}$, $\hat{\mu}:=\frac{1}{\|\mu-\nu\|_{TV}}|\mu-\nu|$ yields
\[
W_{\sq}(\mu,\nu)\le 2^{1-1/\sq}(M_q\|\mu-\nu\|_{TV})^{1/\sq}.
\]
\end{lemma}

\begin{proof}
We consider $\mu_m:=\mu \wedge \nu$, which is defined by
\[
\mu_m(E)=\min(\mu(E), \nu(E)),~\forall E \text{~measurable}.
\]
Define two measures
$\mu_1:=\mu-\mu_m$ and $\nu_1:=\nu-\nu_m$.
Then, 
\begin{gather}
\|\mu-\nu\|_{TV}=\|\mu_1\|_{TV}+\|\nu_1\|_{TV},~~\mu_1+\nu_1\le \delta\,\hat{\mu}.
\end{gather}

Construct the joint distribution (noting $\|\mu_1\|_{TV}=\|\nu_1\|_{TV}$)
\[
d\pi:=\pi(dx, dy)=\frac{1}{\|\mu_1\|_{TV}}\mu_1(dx)\otimes \nu_1(dy)+Q_{\#}\mu_m (dx, dy),
\]
with $Q(x)=(x, x)$ and $Q_{\#}$ is the standard pushforward map as in \eqref{eq:pushforward}. Clearly, the marginal distributions of $\pi$ are $\mu$ and $\nu$ respectively. 

Then, fix $x_0\in \R^d$.
\[
\begin{split}
\int_{\R^d\times\R^d} |x-y|^{\sq} d\pi &=\frac{1}{\|\mu_1\|_{TV}}\int_{\R^d\times\R^d} |x-y|^{\sq} \mu_1(dx)\otimes \nu_1(dy) \\
&\le \frac{2^{\sq-1}}{\|\mu_1\|_{TV}}\int_{\R^d\times\R^d} (|x-x_0|^{\sq}+|y-x_0|^{\sq})\mu_1(dx)\otimes \nu_1(dy)\\
& = 2^{\sq-1}
\int_{\R^d}|x-x_0|^{\sq} [\mu_1+\nu_1](dx) .
\end{split}
\]
Noting $\mu_1+\nu_1\le \delta\, \hat{\mu}$, the claim follows by taking infimum on $x_0$.
\end{proof}

\section{The mean field limit of RBM with \texorpdfstring{$\tau$}{Lg} fixed}\label{sec:pfmeanfield}

\begin{figure}
\begin{center}
	\includegraphics[width=0.8\textwidth]{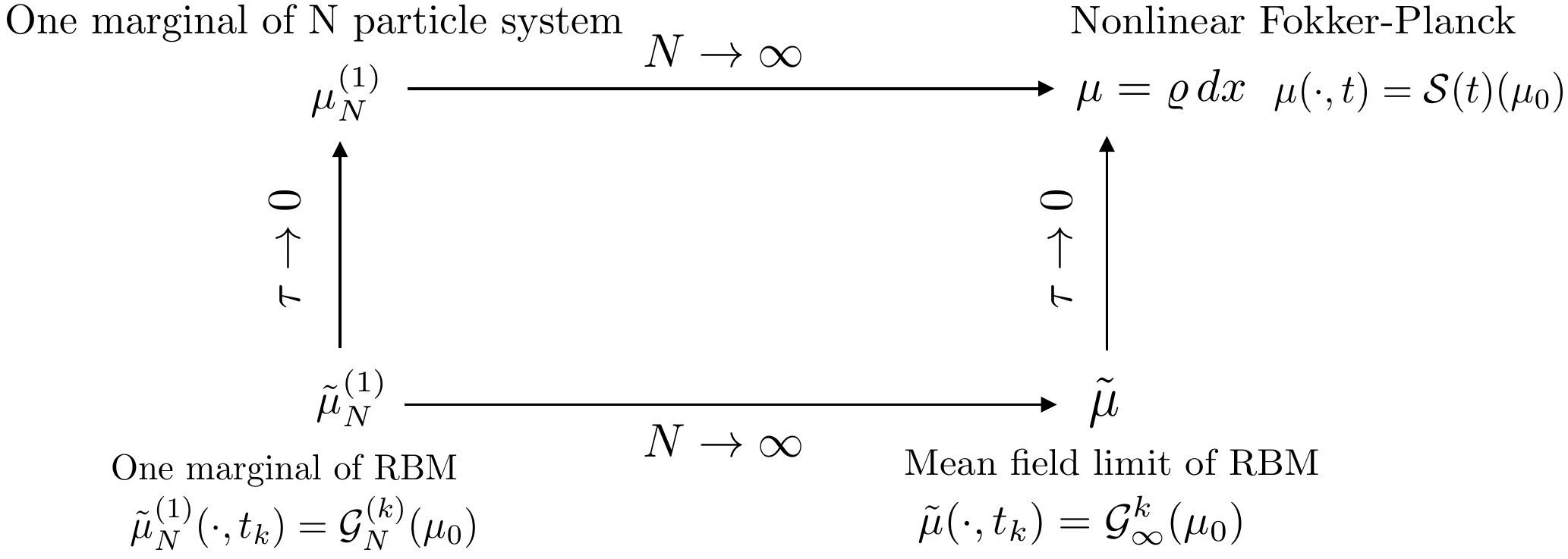}
\end{center}
\caption{Illustration of the various operators and the asymptotic limits.}
\label{fig:operator}
\end{figure}

Starting with $\mu_0$, after $k$ steps of the dynamics given in \eqref{eq:Ginfty}, one arrives at 
\[
\mathcal{G}_{\infty}^{k}(\mu_0)=\mathcal{G}_{\infty}\circ \cdots\circ \mathcal{G}_{\infty}(\mu_0),~  \text{($k$ copies)}, 
\]
which is expected to be the mean field limit of RBM after $k$ steps. Corresponding to this, one may define the operator $\mathcal{G}_N^{(k)}: \bfP(\R^d)\to \bfP(\R^d)$
for RBM with $N$ particles as follows.
Let $X^i_0$'s be i.i.d., drawn from $\mu_0$. Consider \eqref{eq:rbmSDE} and define
\begin{gather}
\mathcal{G}_N^{(k)}(\mu_0):=\mathscr{L}(X^1(t_k)),
\end{gather}
where recall that $\mathscr{L}(X^1)$ means the law of $X^1$, thus the one marginal distribution. Conditioning on a specific sequence of random batches, the particles are not exchangeable. However,
when one considers the mixture of all possible sequences of random batches, the laws of the particles $X^i(t_k)$ ($1\le i\le N$) are identical. In Fig. \ref{fig:operator}, we illustrate these definitions and various limits.

The semigroup property is closely related to the Markovian property. For the $\mathcal{G}_{\infty}$ dynamics, knowing the marginal distribution of $X^1$ can fully determine the probability transition. However, knowing only the marginal distribution is not enough for $\mathcal{G}_{N}^{(k)}$ dynamics, and the joint distribution must be known. Hence, we remark that
\begin{lemma}
 $\{\mathcal{G}_{\infty}^k: k\ge 1\}$ forms a nonlinear semigroup while
$\{\mathcal{G}_N^{(k)}: k\ge 1\}$ is not a semigroup.
\end{lemma}

We first of all introduce some concepts. 
For each particle $i$, we define a sequence of lists $\{L_i^{(k)}: k\ge 0\}$ associated with $i$, given as follows:
\begin{enumerate}[(a)]
\item $L_i^{(0)}=\{i\}$.
\item For $k\ge 1$, let $C_q^{(k-1)}$ be the batch that particle $i$ stays in for $t\in [t_{k-1}, t_k)$. Then,
\begin{gather}
L_i^{(k)}=\cup_{j\in C_q^{(k-1)}}L_j^{(k-1)}.
\end{gather}
\end{enumerate}

\begin{figure}
\begin{center}
	\includegraphics[width=0.8\textwidth]{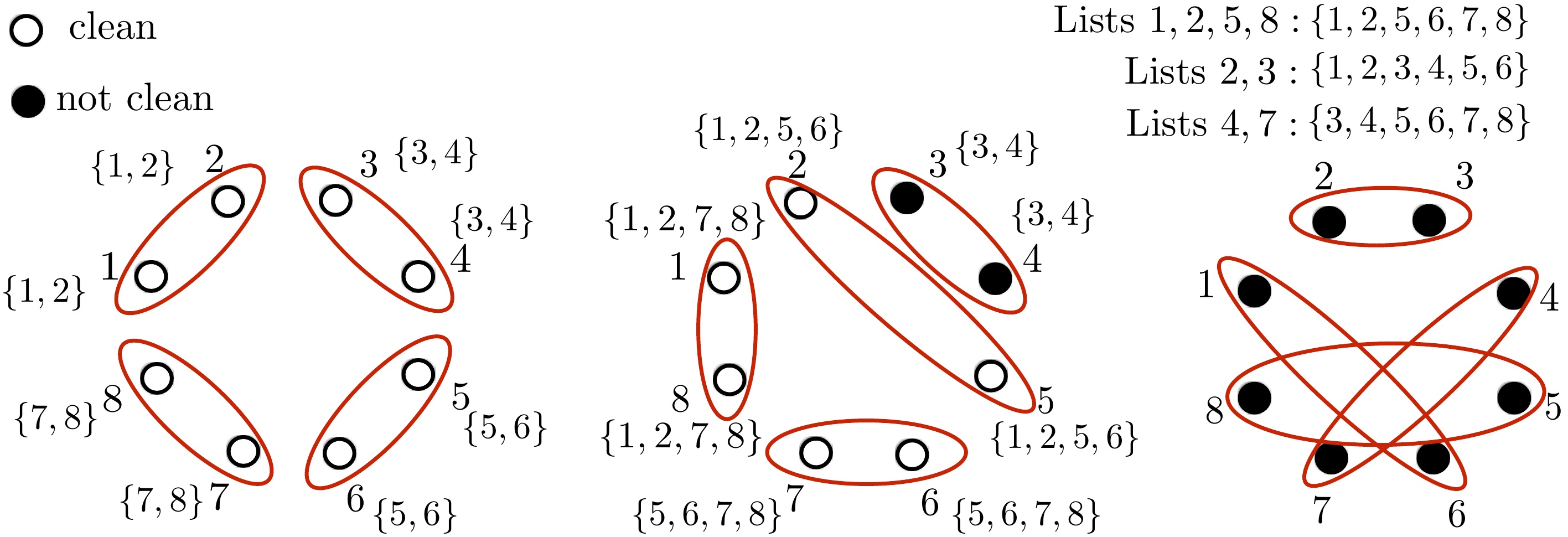}
\end{center}
\caption{Illustration of the definitions of $L_i^{(k)}$ and particles being clean. The three pictures are for $t_1^-, t_2^-, t_3^-$ respectively with $N=8$, $p=2$. The lists (i.e., $\{1, 2\}, \{1,2,5,6\}$ etc) indicate $L_i^{(k)}$ for the corresponding particles.}
\label{fig:clean}
\end{figure}
Here, $L_i^{(k)}$ can be viewed as the particles that have impacted $i$ for $t<t_k$.
Clearly, a particle $i_1\in L_i^{(k)}$ might not have been a batchmate of $i$. It could have been a batchmate of $i_2$, and then $i_2$ was a batchmate of $i$ at some time. The important observation is that if $L_i^{(k)}$ and $L_j^{(k)}$ do not intersect for a given sequence of random batches, then particles $i$ and $j$ are independent at $t_k^-$. Note that we are not claiming all particles in $L_j^{(k)}$ are independent of those in $L_j^{(k)}$ at $t_k^-$. In fact, it is possible that some $i_1\in L_i^{(k)}$ and $j_1\in L_j^{(k)}$  are in the same batch on $[t_{k-1}, t_k)$. However, $i_1$ and $j_1$ must be independent at the times when they were added to the batches that eventually impact $i,j$ at $t_k^-$.
This motivates us to define the following.
\begin{definition}
We say particle $i$ is clean on $[t_k, t_{k+1})$ if the batch $\mathcal{C}_q^{(k)}$ that contains $i$ at $t_k^+$ satisfies the following: (1) any $j\in \mathcal{C}_q^{(k)}$ is clean at $t_{k}^-$; (2) any $j,\ell\in \mathcal{C}_q^{(k)}$ with $j\neq \ell$, $L_j^{(k)}$ and $L_{\ell}^{(k)}$ do not intersect.
\end{definition}

Fig. \ref{fig:clean} gives the illustration for the definitions of $L_i^{(k)}$ and particles being clean. Plainly speaking, a particle $i$ is ``clean" at $t_k^-$ if its batchmates at $t<t_k$ were mutually independent and independent to $i$ when they interacted.

Let us use the symbol $|A|$ below for a set $A$ to mean the cardinality of $A$. The following observation is useful for our argument later.
\begin{lemma}\label{lmm:cleandistribution}
Consider a fixed sequence of divisions of random batches $\{\mathcal{C}^{(\ell)}\}_{\ell\le k-1}$. 
\begin{enumerate}[(i)]
\item It holds that
\[
|L_i^{(k)}|\le p^k,
\]
and the particle $i$ is clean at $t_k^-$ if and only if the equality holds. 
\item The distribution of $X^i$ for a clean particle $i$ at $t_k^-$ is $\mathcal{G}_{\infty}^k (\mu_0)$.
\end{enumerate}
\end{lemma}
\begin{proof}
The proof is a straightforward induction. Here, let us just mention the proof of the second claim briefly.

For $k=0$, the statement is trivial.  Now, suppose the statement is true for all $k\le m-1$. We now consider $k=m$.

For the given sequence of random batches $\{\mathcal{C}^{(\ell)}\}_{\ell\le k-1}$, that a particle $i$  is clean at $t_{m}^-$
means that on $[t_{m-1}, t_m)$, the particles  in the batch for $i$ are independent of $i$ at $t_{m-1}$.
By the induction assumption, the distribution of one particle at $t_{m-1}$ is given by $\tilde{\mu}(\cdot, t_{m-1})=\mathcal{G}_{\infty}^{m-1}(\mu_0)$.  By the independence, the joint distribution of them at $t_{m-1}$ is therefore 
\[
\rho^{(p)}(\cdots, t_{m-1})=\tilde{\mu}(\cdot, t_{m-1})^{\otimes p}.
\]
From $t_{m-1}$ to $t_m$, the evolution of the joint distribution obeys the Fokker-Planck equation \eqref{eq:firstalgorithm}.
Hence, at $t_m^-$, the distribution of particle $i$ is given by $\mathcal{G}_{\infty}^m (\mu_0)$ by the definition (equation \eqref{eq:Ginfty}).
\end{proof}

Let $A_k$ denote the set of particles that are clean at $t_k^-$. Then, 
\[
A_1=A_0=\{1, \cdots, N\}.
\]
For $k\ge 2$, one has
\begin{gather}
\begin{split}
A_k=\Big\{i\in A_{k-1}: i\in \mathcal{C}_q^{(k-1)},~&\forall j, \ell \in \mathcal{C}_q^{(k-1)},j\neq \ell, \\
~& j\in A_{k-1},\ell\in A_{k-1}, L_j^{(k-1)}\cap L_{\ell}^{(k-1)}=\emptyset \Big\}.
\end{split}
\end{gather}

Denote
\begin{gather}
\epsilon_{k}:=\mathbb{P}(1\notin A_k).
\end{gather}
Note that by symmetry, $\epsilon_k$ is also the probability that particle $i$ is not clean. We state our main result.
\begin{theorem}\label{thm:meanfield}
Let $\sq\in [1,\infty)$. It holds that
\begin{gather}
W_{\sq}(\mathcal{G}_{\infty}^k(\mu_0), \mathcal{G}_N^{k}(\mu_0))\le C\exp(\alpha t_k)\epsilon_{k}^{1/\sq},
\end{gather}
for some $\alpha>0$. In the strong confinement case $\alpha=0$.
\end{theorem}

To prove Theorem \ref{thm:meanfield}, we need some preparation.  We first establish some moments estimates.
\begin{lemma}\label{lmm:momentsoflimitingdynamics}
If Assumption \ref{ass:momentmu0} holds, then for $q\ge 1$, there exists $\alpha_1(q)\ge 0$ such that
\begin{gather}
\sup_{k: k\tau\le T}\int_{\R^d} |x|^q\mathcal{G}_{\infty}^k(\mu_0)(dx)\le C(q)e^{\alpha_1(q) T}.
\end{gather}
In the strong confinement case (Assumption \ref{ass:kernelstrong}), one can take $\alpha_1(q)=0$, i.e., the constant in the upper bound can be uniform in $T$.
\end{lemma}

\begin{proof}

We note that $\{\mathcal{G}_{\infty}^k\}$ is a semigroup, so it suffices to estimate the growth of the moments in one step.

First, consider \eqref{eq:RBMmeanfieldSDE} and take $q\ge 2$.  By It\^o's calculus, one has
\begin{multline}
d\E|Y^i|^q=q\E|Y^i|^{q-2}Y^i\cdot\left[b(Y^i)+\frac{1}{p-1}\sum_{j=1,j\neq i}^{p}K(Y^i-Y^j)\right]\,dt\\
+\E q|Y^i|^{q-2}(d+q-2)\sigma^2\,dt.
\end{multline}

Using the one-sided Lipschitz condition in Assumption \ref{ass:kernelfunctions}, one has
\begin{gather*}
Y^i\cdot b(Y^i)=(Y^i-0)\cdot(b(Y^i)-b(0))+Y^i\cdot b(0)
\le \beta |Y^i|^2+C|Y^i|.
\end{gather*}
Similarly, since $K$ is Lipschitz, one has $|K(Y^i-Y^j)|\le |K(0)|+L(|Y^i|+|Y^j|)$, and thus
\begin{gather*}
Y^i\cdot K(Y^i-Y^j)
\le |K(0)||Y^i|+L(|Y^i|^2+|Y^i||Y^j|).
\end{gather*}
It follows that
\begin{gather*}
\begin{split}
&\E|Y^i|^{q-2}Y^i\cdot\left[b(Y^i)+\frac{1}{p-1}\sum_{j=1,j\neq i}^{p}K(Y^i-Y^j)\right]\\
& \le (\beta+L)\E|Y^i|^q+C\E |Y^i|^{q-1}
+\frac{1}{p-1}\sum_{j: j\neq i}\E|Y^i|^{q-1}|Y^j|.
\end{split}
\end{gather*}

By Young's inequality,
\[
 \E|Y^i|^{q-1}|Y^j| \le \frac{(q-1)\nu}{q}\E|Y^i|^q+\frac{\E |Y^j|^q}{q\nu^{q-1}}
\]
for any $\nu>0$. In particular, one also has
\[
\E|Y^i|^{q-1}\le \frac{(q-1)\nu}{q}\E|Y^i|^q+\frac{1}{q\nu^{q-1}}.
\]
Similarly, using Young's inequality, $\E|Y^i|^{q-2}$ is also easily controlled by
$\delta \E|Y^i|^q+C(\delta)$ for some small $\delta$.

By the exchangeability so that $\E|Y^i|^q=\E|Y^j|^q$, one then has
\begin{gather*}
\frac{d}{dt}\E|Y^i|^q\le q(\beta+2L+\delta)\E|Y^i|^q+C_2.
\end{gather*}

In the strong confinement case as in Assumption \ref{ass:kernelstrong},
\begin{gather*}
\begin{split}
&\E|Y^i|^{q-2}X^i\cdot\left[b(Y^i)+\frac{1}{p-1}\sum_{j=1,j\neq i}^{p}K(Y^i-Y^j)\right]\\
& \le (-r+L)\E|Y^i|^q+C\E|Y^i|^{q-1}
+\frac{L}{p-1}\sum_{j: j\neq i}\E|Y^i|^{q-1}|Y^j|\\
&\le (-r+L+\frac{(q-1)L}{q})\E|Y^i|^q+\frac{L}{p-1}\sum_{j: j\neq i}\frac{1}{q}\E|Y^j|^q
+\delta\E|Y^i|^{q}+C(\delta)\\
&=(-r+2L)\E|Y^i|^q+\delta\E|Y^i|^{q}+C(\delta),
\end{split}
\end{gather*}
where $\delta$ is a sufficiently small but fixed number. The conclusions then follow easily for $q\ge 2$.

If $q\in [1,2)$, one then uses the H\"older inequality $(\E|Y^i|^q)^{1/q}
\le (\E|Y^i|^{r})^{1/r}$ for $r\ge q$ to get the desired result.
\end{proof}

We also need the moment control for the Random Batch Method 
conditioning on any specific sequence of random batches.
\begin{lemma}\label{lmm:momentsofrbm}
Consider a fixed sequence of divisions of random batches $\{\mathcal{C}^{(\ell)}\}$. 
Again, consider the solutions $\{X^i(t)\}_{i=1}^N$ to \eqref{eq:firstalgorithm}. 
Then,
\begin{gather}
\sup_{t\le T}\sup_{i}\E\left(|X^i|^q |  \{\mathcal{C}^{(\ell)}\}\right)\le C(q)e^{\alpha_1 T}.
\end{gather}
In the strong confinement Assumption \ref{ass:kernelfunctions},
\begin{gather}
\sup_{t\ge 0}\sup_{i}\E\left(|X^i|^q | \{\mathcal{C}^{(\ell)}\}\right)\le C(q),
\end{gather}
where $C(q)$ and $\alpha_1$ do not depend on the specific sequence of divisions of random batches $\{\mathcal{C}^{(\ell)}\}$.
\end{lemma}

\begin{proof}
The proof follows the same line as that in Lemma \ref{lmm:momentsoflimitingdynamics}.
The difference is that there is no exchangeability now conditioning on the random batches.

Under Assumption \ref{ass:kernelfunctions} and using the similar estimates as in 
Lemma \ref{lmm:momentsoflimitingdynamics}, one has for $t\in [t_k, t_{k+1}]$ that
\begin{gather}\label{eq:conditionalaux1}
\begin{split}
\frac{d}{dt}\E(|X^i|^q | \{\mathcal{C}^{(\ell)}\})
\le q\left(\beta+L+\frac{(q-1)L}{q}+\delta\right)\E(|X^i|^q | \{\mathcal{C}^{(\ell)}\})\\
+\frac{qL}{p-1}\sum_{j: j\neq i}\frac{1}{q}\E(|X^j|^q | \{\mathcal{C}^{(\ell)}\})
+C(\delta).
\end{split}
\end{gather}

Under Assumption \ref{ass:kernelstrong}, one then has for $t\in [t_k, t_{k+1}]$ that
\begin{gather}\label{eq:conditionalaux2}
\begin{split}
\frac{d}{dt}\E(|X^i|^q | \{\mathcal{C}^{(\ell)}\})
\le q\left(-r+L+\frac{(q-1)L}{q}+\delta \right)\E(|X^i|^q | \{\mathcal{C}^{(\ell)}\})\\
+\frac{qL}{p-1}\sum_{j: j\neq i}\frac{1}{q}\E(|X^j|^q | \{\mathcal{C}^{(\ell)}\})
+C(\delta).
\end{split}
\end{gather}

Next, based on \eqref{eq:conditionalaux1}, one easily finds
\begin{gather*}
\begin{split}
\E(|X^i(t)|^q | \{\mathcal{C}^{(\ell)}\})-\E(|X^i|^q(t_k) | \{\mathcal{C}^{(\ell)}\}) 
\le &q\left(\beta+L+\frac{(q-1)L}{q}+\delta\right)\int_{t_k}^t \E(|X^i(s)|^q | \{\mathcal{C}^{(\ell)}\})\,ds\\
&+L\int_{t_k}^t \max_{1\le i\le p}\E(|X^i(s)|^q | \{\mathcal{C}^{(\ell)}\})\,ds+\int_{t_k}^tC(\delta)\,ds.
\end{split}
\end{gather*}
It follows that
\begin{gather}\label{eq:a}
a(t):=\max_{1\le i\le p}\E(|X^i|^q | \{\mathcal{C}^{(\ell)}\})
\end{gather}
satisfies
\[
a(t)\le a(t_k)+q(\beta+2L+\delta)\int_{t_k}^t [a(s)+C(\delta)]\,ds.
\]
Gr\"onwall's inequality then yields the first claim with any
$\alpha>\beta+2L$.

For \eqref{eq:conditionalaux2}, defining $r_1:=q\left(r-L-\frac{q-1}{q}L-\delta \right)>0$, one finds that
\begin{multline*}
\E(|X^i(t)|^q | \{\mathcal{C}^{(\ell)}\})
\le \E(|X^i|^q(t_k) | \{\mathcal{C}^{(\ell)}\}) e^{-r_1(t-t_k)} \\
+\int_{t_k}^t e^{-r_1(t-s)}\left[L\max_{1\le i\le p}\E(|X^i(s)|^q | \{\mathcal{C}^{(\ell)}\})+C(\delta)\right]\,ds.
\end{multline*}
Hence, the function $a$ defined in \eqref{eq:a} satisfies
\[
a(t)\le a(t_k)e^{-r_1(t-t_k)}
+\int_{t_k}^t e^{-r_1(t-s)}\left[La(s)+C(\delta)\right]\,ds.
\]
It can be shown easily that $a(t)$ is controlled by $b(t)$ which satisfies the following
integral equality
\[
b(t)=  a(t_k)e^{-r_1(t-t_k)}
+\int_{t_k}^t e^{-r_1(t-s)}\left[Lb(s)+C(\delta)\right]\,ds.
\]
(One can perturb the initial data $a(t_k)\to a(t_k)+\epsilon$ for $b(\cdot)$ and then take $\epsilon\to 0$).

Then, one finds
\[
b'(t)=(-r_1+L)b(t)+C(\delta)=q(-r+2L+\delta)b(t)+C(\delta),~~~b(t_k)=a(t_k).
\]
Hence,
\[
a(t_{k+1})\le b(t_{k+1})\le a(t_k)e^{-q(r-2L+\delta)\tau}+\frac{C(\delta)}{q(r-2L-\delta)}
(1-e^{-q(r-2L+\delta)\tau}).
\]
The second claim also follows.
\end{proof}

Now, we can prove the main theorem in this section.
\begin{proof}[Proof of Theorem \ref{thm:meanfield}]

First of all, for the $N$-particle system, by symmetry, the distribution of any particle is equal to $\mathcal{G}_{N}^k(\mu_0)$. Now, we focus on a particular particle $i=1$, for example.

By Lemma \ref{lmm:cleandistribution},
\begin{gather}\label{eq:GNdecompose}
\mathcal{G}_N^k(\mu_0)=\bbP(1\in A_k)\mathcal{G}_{\infty}^k(\mu_0)
+\bbP(1\notin  A_k)\nu_k,
\end{gather}
for some probability measure $\nu_k$. 
To see this, we consider all possible sequences of random batches. Only the first $k$ divisions of batches (i.e. ones at $t_0,\cdots, t_{k-1}$) will affect the distribution at $t_k$. This subsequence (the first $k$ divisions) can take only finitely many values, and let $\{c^{\ell}\}_{\ell\le k-1}$ be such values. Then, for any  $E\subset\R^d$ that is Borel measurable,
\[
\mathcal{G}_N^k(\mu_0)[E]=\sum_{\{\mathcal{C}^{\ell}=c^{\ell},\ell\le k-1\}} \mathbb{P}(\mathcal{C}^{\ell}=c^{\ell},\ell\le k-1)
\mathbb{P}(X^1 \in E | \mathcal{C}^{\ell}=c^{\ell},\ell\le k-1).
\]
Lemma \ref{lmm:cleandistribution} tells us that if $\{c^{\ell}\}_{\ell\le k-1}$ is a value such that $1$ is clean, then 
\[
\mathbb{P}(X^1\in E | \mathcal{C}^{\ell}=c^{\ell},\ell\le k-1)=\mathcal{G}_{\infty}^k(\mu_0)[E].
\]
Hence,
\begin{multline*}
\mathcal{G}_N^k(\mu_0)[E]=\bbP(1\in A_k)\mathcal{G}_{\infty}^k(\mu_0)[E]\\
+\sum_{\{\mathcal{C}^{\ell}=c^{\ell},\ell\le k-1\},1\notin A_k}\mathbb{P}(\mathcal{C}^{\ell}=c^{\ell},\ell\le k-1)
\mathbb{P}(X^1\in E | \mathcal{C}^{\ell}=c^{\ell},\ell\le k-1)\\
=\bbP(1\in A_k)\mathcal{G}_{\infty}^k(\mu_0)[E]+\bbP(1\notin A_k)\nu_k(E),
\end{multline*}
with
\begin{equation*}
\nu_k(E)=\sum_{\{\mathcal{C}^{\ell}=c^{\ell},\ell\le k-1\},1\notin A_k}\frac{\mathbb{P}(\mathcal{C}^{\ell}=c^{\ell},\ell\le k-1)}{\bbP(1\notin  A_k)}
\mathbb{P}(X^1\in E | \mathcal{C}^{\ell}=c^{\ell},\ell\le k-1).
\end{equation*}
Clearly, $\nu_k$ is a convex combination of some conditional marginal distributions of $X^1$, each being $\mathscr{L}(X^1)$ conditioning on a particular sequence of batches for $\{1\notin A_k\}$. Hence, $\nu_k$ is a probability measure.

By \eqref{eq:GNdecompose}, it holds that
\begin{gather}
|\mathcal{G}_{\infty}^k(\mu_0)-\mathcal{G}_N^k(\mu_0)|
\le (1-\bbP(1\in A_k)) \mathcal{G}_{\infty}^k(\mu_0)+\bbP(1\notin A_k)\nu_k=\epsilon_k(\mathcal{G}_{\infty}^k(\mu_0)+\nu_k).
\end{gather}
Therefore, the total variation distance between the two measures is controlled by
\begin{gather}
\|\mathcal{G}_{\infty}^k(\mu_0)-\mathcal{G}_N^k(\mu_0)\|_{TV}
\le (1-\bbP(1\in A_k))+\bbP(1\notin A_k)= 2\epsilon_k.
\end{gather}

By Lemma \ref{lmm:momentsofrbm}, for each sequence of batches, one has
\[
\sup_i \E (|X^i|^q| \mathcal{C}^{(\ell)})\le C(q)e^{\alpha_1 t}.
\]
Hence, it holds that
\begin{gather}\label{eq:pfmoment1}
\int_{\R^d} |x|^q \nu_k(dx) \le C(q)e^{\alpha_1 t}.
\end{gather}
In the case of strong confinement, $\alpha_1=0$.
Similarly, by Lemma \ref{lmm:momentsoflimitingdynamics}, $\mathcal{G}_{\infty}^k(\mu_0)$ has the same moment control. Application of Lemma \ref{lmm:W2byTV} then yields the desired result.
\end{proof}

Lastly, we close up the estimate.
\begin{theorem}
For any fixed $k$, it holds that
\begin{gather}
\lim_{N\to\infty}\epsilon_k=0.
\end{gather}
\end{theorem}

\begin{proof}

First of all, clearly, we have
\[
\epsilon_0=\epsilon_1=1-1=0.
\]

Now, we do induction on $k$. Assume
\[
\lim_{N\to\infty}\epsilon_k=0.
\]
Consider the batches for $t_k \to t_{k+1}^-$.
Assume the batch for particle $1$ is $\mathcal{C}_q^{(k)}$.
Denote
\[
B_k=\Big\{\forall j, \ell\in \mathcal{C}_q^{(k)}, j\neq \ell, L_j^{(k)}\cap L_{\ell}^{(k)}=\emptyset \Big\}.
\]

Let $\mathcal{B}=\mathcal{C}_q^{(k)} \setminus \{1\}$ be the set of other particles that share the same batch with particle $1$. Then, by definition of $A_{k+1}$,
\begin{multline}
\mathbb{P}(1\in A_{k+1})
=\sum_{j_1,\cdots, j_{p-1}}\bbP(\mathcal{B}=\{j_1,\cdots, j_{p-1}\})\times \\
 \mathbb{P}(B_k\cap \{1 \in A_k\}\cap_{\ell=1}^{p-1}\{j_{\ell}\in A_k\} | \mathcal{B}=\{j_1,\cdots, j_{p-1}\} ).
\end{multline}

Denote $E:=\{\mathcal{B}=\{j_1,\cdots, j_{p-1}\}\}$, where we omit the dependence in $j_{\ell}, 1\le \ell\le p-1$ for notational convenience.
Conditioning on $B_k\cap E $ (i.e., provided that the event $B_k\cap E$ happens), whether the particles are clean or not are independent. Hence,
\begin{multline*}
\bbP(B_k\cap \{1 \in A_k\}\cap_{\ell=1}^{p-1}\{j_{\ell}\in A_k\} | E)
=\bbP(B_k| E)\bbP(\{1 \in A_k\}\cap_{\ell=1}^{p-1}\{j_{\ell}\in A_k\} | E, B_k)\\
=\bbP(B_k | E) \prod_{\ell=1}^{p}\bbP(j_{\ell}\in A_k | E, B_k),
\end{multline*}
where we have set $j_p=1$. Moreover, 
\begin{multline*}
\bbP(j_{\ell}\in A_k | E, B_k)
=\frac{\bbP(\{j_{\ell}\in A_k\}\cap E\cap B_k)}{\bbP(E\cap B_k)}
=\frac{\bbP(\{1\in A_k\}\cap E\cap B_k)}{\bbP(E\cap B_k)}=\frac{\bbP(\{1\in A_k\}\cap B_k)}{\bbP( B_k)}.
\end{multline*}
The second and the last equalities are due to symmetry. For the last equality, 
$\bbP(\{\mathcal{B}=\{j_1,\cdots, j_{p-1}\}\}\cap B_k)$ should be equal for all possible $j_1,\cdots, j_{p-1}$, and the same is true for the numerator. 
 This actually is a kind of independence.
Hence, eventually due to the fact 
\[
\sum_{j_1,\cdots, j_{p-1}}\bbP(\mathcal{B}=\{j_1,\cdots, j_{p-1}\})\bbP(B_k| \mathcal{B}=\{j_1,\cdots, j_{p-1}\})=\bbP(B_k),
\]
one has
\[
1-\epsilon_{k+1}=\mathbb{P}(1\in A_{k+1})
\ge \mathbb{P}(B_k)(1-\epsilon_k/\mathbb{P}(B_k))^p.
\]

Hence, it suffices to show
\[
\lim_{N\to\infty}\mathbb{P}(B_k)=1.
\]
To get an estimate for this, we consider the following equivalent way to construct $L_i^{(k)}$: one starts with $L_i \leftarrow \{i\}$ and  repeat the following for $k$ times:
\begin{enumerate}[(1)]
\item  Set $L_{\mathrm{tmp}}\leftarrow L_i$ and $A=\emptyset$.
\item  Loop the following until $L_{\mathrm{tmp}}$ is empty.
\begin{enumerate}[(a)]
\item Pick a particle $i_1 \in L_{\mathrm{tmp}}$, then choose $p-1$ particles from $\{1,\cdots, N\}\setminus A\cup\{i_1\}$ denoted by $\{i_2, \cdots, i_{p}\}$. 

\item Set $L_i \leftarrow L_i\cup \{i_2, \cdots, i_p\}$.
\item Set $A\leftarrow A\cup \{i_1, i_2,\cdots, i_p\}$.
\item Set $L_{\mathrm{tmp}}\leftarrow L_{\mathrm{tmp}}\setminus \{i_1, i_2, \cdots, i_p\}$.
\end{enumerate}
\end{enumerate}
In the above, we are actually looking back from $t_{k-1}$. In the $j$th iteration,
we are constructing batches at $t_{k-j}$. Hence, this is an equivalent way to construct $L_i^{(k)}$.

Now, we estimate $\mathbb{P}(B_k)$ by constructing the lists $L_{j_{\ell}}^{(k)}: 1\le \ell\le p$ for $j_{\ell}\in C_q^{(k)}$ using the above procedure. Consider that the lists for $j_{1},\cdots, j_{\ell-1}$ have been constructed, which have included at most
$(\ell-1)p^k$ particles. Now, for $L_{j_{\ell}}^{(k)}$ not to intersect
with the previous lists, one has to choose particles
from $\{1,\cdots, N\}\setminus [\cup_{z=1}^{\ell-1}L_{j_z}^{(k)}\cup A\cup \{i_1\}]$
in 2(a) step. Conditioning on the specific choices of $L_{j_{\ell}}^{(k)}: 1\le \ell\le p$ and $A$ with
\[
N_1:=|L_{j_{\ell}}^{(k)}\cup A\cup \{i_1\}|,~~N_2:=|A|,
\]
this probability is controlled from below by
\[
\frac{{{N-N_1}\choose{p-1}}}{{{N-1-N_2}\choose{p-1}}} \ge \frac{{{N-1-\ell p^k}\choose{p-1}}}{{{N-1}\choose{p-1}}}.
\]
Hence, as $N\to\infty$,
\[
\bbP(B_k)\ge \prod_{\ell=1}^{p}\left[\frac{{{N-1-\ell p^k}\choose{p-1}}}{{{N-1}\choose{p-1}}}\right]^k=1-O(N^{-1}).
\]
Hence, $\lim_{N\to\infty}\epsilon_{k+1}=0$ and the claim follows.
\end{proof}
As can be seen in the proof, one actually has $\epsilon_k\le C(p,k)N^{-1}$ for some $C(p,k)>0$. This rate is different from the typical $O(N^{-1/2})$ rate (though under $W_2$ distance) for the mean field limit of interacting particle systems due to law of large number results.

\begin{remark}
The current argument of the mean field limit relies on the fact that two particles are unlikely to be related when $N\to\infty$ for finite iterations. This is not enough to get the mean field limit independent of $\tau$. For fixed $N$, $\epsilon_{k}\to 1$ as $k\to\infty$. As pointed in \cite{jin2020random}, RBM works due to the averaging effect in time. The regime we consider here (finite iterations and $N\to \infty$) is clearly far before the averaging effect in time comes into play. To consider the mean field limit uniform in $\tau$ (the averaging mechanism can take effect), one must consider carefully how the correlation decays as $k$ grows when two particles are not totally clean to each other. The study of this creation of chaos will be left for the future. 
\end{remark}

\section{Properties of the limiting dynamics }\label{sec:tautozero}

We consider the limit dynamics given by the operator $\mathcal{G}_{\infty}$ (defined in \eqref{eq:Ginfty}) and its approximation to the dynamics of the nonlinear Fokker-Planck equation \eqref{eq:nonlinearFP}, the mean-field limit of the interacting particle system \eqref{eq:Nparticlesys}.

As proved in \cite{jin2020random}, the error between the one marginal distribution of the RBM particle system \eqref{eq:firstalgorithm} and that of \eqref{eq:Nparticlesys} are close independent of $N$ under $W_2$ distance (the left side in Fig. \ref{fig:operator}). Combining the mean field result in section \ref{sec:pfmeanfield} and taking $N\to\infty$,  one sees that the dyanmics of $\mathcal{G}_{\infty}$ is close to that of \eqref{eq:nonlinearFP} (the right side in Fig. \ref{fig:operator}). In other words, the two limits $\lim_{N\to\infty}$ and $\lim_{\tau \to 0}$ commute.

A direct application of the strong mean square error in \cite{jin2020random} gives an upper bound $O(\sqrt{\tau})$ for the $W_2$ distance corresponding to the left side in Fig. \ref{fig:operator}, and thus the right side in Fig. \ref{fig:operator} after taking $N\to\infty$.  It is shown in \cite{jin2020convergence}  (though for $b(\cdot)$ being bounded) that the weak error is $O(\tau)$. The $W_{\sq}$ distance is a kind of weak topology as it measures the closeness between distributions instead of the trajectories of particles. Hence, the sharp upper bound for the Wasserstein distance between these two marginal distributions is believed to be $O(\tau)$, even for unbounded $b(\cdot)$.  Below, we aim to prove these under $W_1$ distance.

\subsection{Stability of the limiting dynamics}

In this section, we study the stability and contraction properties of the nonlinear operator $\mathcal{G}_{\infty}$ for the limiting dynamics.

\begin{proposition}\label{pro:stabilityW2}
Under Assumption \ref{ass:kernelfunctions},  $\mathcal{G}_{\infty}$ satisfies for $\sq\in [1,\infty)$ that
\begin{gather}\label{eq:limitingW2stable}
W_{\sq}(\mathcal{G}_{\infty}(\mu_1), \mathcal{G}_{\infty}(\mu_2))
\le e^{(\beta+2L)\tau}W_{\sq}(\mu_1, \mu_2),~~\mu_i\in \bfP(\R^d),~i=1,2.
\end{gather}

The operator $\mathcal{G}_{\infty}$ is a contraction in $W_{\sq}$ under Assumption \ref{ass:kernelstrong}:
\begin{gather}
W_{\sq}(\mathcal{G}_{\infty}(\mu_1), \mathcal{G}_{\infty}(\mu_2))
\le e^{-(r-2L) \tau} W_{\sq}(\mu_1, \mu_2),
\end{gather}
so that $\mathcal{G}_{\infty}$ has a unique invariant measure $\pi_{\tau}$ and 
it holds for any $\mu_0$ that
\begin{gather}
W_{\sq}(\mathcal{G}_{\infty}^n(\mu_0), \pi_{\tau})\le e^{-(r-2L)n\tau} W_{\sq}(\mu_0, \pi_{\tau}).
\end{gather}

\end{proposition}

\begin{proof}

Consider two copies of \eqref{eq:RBMmeanfieldSDE}: one is
\begin{gather}
dY_1^i=b(Y_1^i)\,dt+\frac{1}{p-1}\sum_{j=1,j\neq i}^{p}K(Y_1^i-Y_1^j)\,dt
+\sqrt{2}\sigma\,dW^i,~~i=1,\cdots, p,
\end{gather}
with $(Y_1^1(0), \cdots, Y_1^p(0))$ being drawn from $\mu_1^{\otimes p}$;
the other one is
\begin{gather}
dY_2^i=b(Y_2^i)\,dt+\frac{1}{p-1}\sum_{j=1,j\neq i}^{p}K(Y_2^i-Y_2^j)\,dt
+\sqrt{2}\sigma\,dW^i,~~i=1,\cdots, p,
\end{gather}
with $(Y_2^1(0), \cdots, Y_2^p(0))$ being drawn from $\mu_2^{\otimes p}$.

For any $\epsilon>0$, choose the coupling as follows. First, choose a coupling $\gamma$ for $Y_1^1$ and $Y_2^1$ such that
\[
\E|Y_1^1(0)-Y_2^1(0)|^{\sq}\le \epsilon+W_{\sq}^{\sq}(\mu_1, \mu_2).
\]
Then, let the samples $(Y_1^i(0), Y_2^i(0))$ be i.i.d., drawn from $\gamma$. 
Let the Brownian motions for the two systems be the same. 

Now, to show the claims, it suffices to show that the moments of the SDE system \eqref{eq:RBMmeanfieldSDE} are stable. In fact, the joint distribution of $(Y_1^1(\tau), Y_2^1(\tau))$ is a coupling for $\mathcal{G}_{\infty}(\mu_1)$
and $\mathcal{G}_{\infty}(\mu_2)$:
\[
W_{\sq}(\mathcal{G}_{\infty}(\mu_1), \mathcal{G}_{\infty}(\mu_2)) 
\le (\E|Y_1^1(\tau)-Y_2^1(\tau)|^{\sq})^{1/\sq}.
\]

Using the symmetry, it can be computed directly that under Assumption \ref{ass:kernelfunctions}
\[
\frac{d}{dt}\E |Y_1^1-Y_2^1|^{\sq}
\le \sq (\beta+2L) \E|Y_1^1-Y_2^1|^{\sq},
\]
and that under Assumption \ref{ass:kernelstrong}
\[
\frac{d}{dt}\E |Y_1^1-Y_2^1|^{\sq}
\le \sq (-r+2L)\E|Y_1^1-Y_2^1|^{\sq}.
\]
For $\sq=1$, one can use $\sqrt{|Y_1^1-Y_2^1|^2+\delta}$ to approximate and then take $\delta\to 0^+$. Applying Gr\"onwall's inequality and noticing $\e$ is arbitrary, one obtains the first two assertions directly. 
The last claim follows from the standard contraction mapping theorem \cite[Chap. 1]{granas2013}.
\end{proof}

\subsection{Basic properties of the nonlinear Fokker-Planck equation}

We establish several basic results to  \eqref{eq:nonlinearFP} using a stronger version of Assumption \ref{ass:momentmu0}:
\begin{assumption}\label{ass:newrho0}
The measure $\mu_0$ has a density $\varrho_0$ that is smooth with finite moments
$\int_{\R^d}|x|^q\varrho_0\,dx<\infty$, $\forall q\ge 1$,
and the entropy is finite
\begin{gather}
H(\mu_0):=\int_{\R^d}\varrho_0\log\varrho_0\,dx<\infty.
\end{gather}
\end{assumption}
If $\varrho_0(x)=0$ at some point $x$, one defines $\varrho_0(x)\log\varrho_0(x)=0$.
We also introduce the following assumption on the growth rate of derivatives of $b$ and $K$, which will be used below.
\begin{assumption}\label{ass:polynomialgrowth}
The function $b$ and its derivatives have polynomial growth.
The derivatives of $K$ with order at least $2$ (i.e., $D^{\alpha}K$ with $|\alpha|\ge 2$) have polynomial growth.
\end{assumption}

Based on these conditions, equation \eqref{eq:nonlinearFP} can be formulated in terms of the density of $\mu$:
\begin{gather}\label{eq:nonFPdensity}
\begin{split}
&\partial_t\varrho=-\nabla\cdot((b(x)+K*\varrho)\varrho)+\sigma^2\Delta\varrho,\\
&\varrho(0)=\varrho_0.
\end{split}
\end{gather}
Then, a weak solution to \eqref{eq:nonFPdensity} corresponds to 
a measure solution $\mu=\varrho\,dx$ to \eqref{eq:nonlinearFP}, where the weak solution is defined as follows.
\begin{definition}
We say $\varrho\in L^{\infty}([0, T]; L^1(\R^d))$ is a weak solution to \eqref{eq:nonFPdensity}, if  $\varrho\,dx\in C([0, T]; \bfP(\R^d))$ where $\bfP(\R^d)$ is equipped with the weak topology, and for any $\varphi\in C_c^{\infty}(\R^d)$, it holds for any $t\le T$ that
\begin{multline}
\int_{\R^d} \varrho(x,t)\varphi(x)\,dx-\int_{\R^d}\varrho_0(x)\varphi(x)\,dx\\
=\int_0^t\int_{\R^d} \nabla\varphi(x)\cdot(b(x)+K*\varrho)\,\rho(x,s)dxds
+\sigma^2\int_0^t\int_{\R^d}\Delta\varphi(x) \varrho(x,s)\,dxds.
\end{multline}
\end{definition}

Note that the test function used here does not depend on time variable, so we require the integral equation to hold for any $t\le T$. Due to the relation between
\eqref{eq:nonFPdensity} and  \eqref{eq:nonlinearFP}, we will not distinguish the measure and its density. For example, we will use $\mathcal{G}_{\infty}(\varrho_0)$
to mean the nonlinear semigroup acting on the measure $\mu_0$, and will use 
$W_{\sq}(\varrho, \nu)$ to mean the Wasserstein-$\sq$ distance between $\mu=\varrho\,dx$ and another measure $\nu$.

We have the following regarding the well-posedness of the nonlinear Fokker-Planck equation \eqref{eq:nonFPdensity}.
\begin{proposition}\label{pro:nonlinearfpestimates}
Let  Assumption \ref{ass:kernelfunctions} or  Assumption \ref{ass:kernelstrong} hold, and also $|b|+|\nabla b|\le C(1+|x|^q)$
for some $C, q$. Fix any $T>0$. Assume the initial data $\varrho_0$ satisfies Assumption \ref{ass:newrho0}. Then, the nonlinear Fokker-Planck equation \eqref{eq:nonFPdensity} has a unique weak solution satisfying
$\sup_{0\le t\le T}\int_{\R^d} |x|\varrho\,dx<\infty$.
Moreover, this solution is a strong solution and is smooth together with the moment control:
\begin{gather}
\sup_{0\le t\le T}\int_{\R^d} |x|^q\varrho\,dx\le C(q,T)\int_{\R^d}|x|^q\varrho_0\,dx.
\end{gather}
Besides, under Assumption \ref{ass:kernelstrong}, the moments are uniformly bounded in $t$, i.e., the constants $C(q,T)$ above can be made independent on $T$.
Moreover, $\mu=\varrho\,dx$ converges in $W_{\sq},\sq\ge 1$ to an invariant measure $\pi$ exponentially as $t\to\infty$.
\end{proposition}

There are many works on similar models in literature, and see \cite{mckean1966class,carrillo2003kinetic,cattiaux2008,barbu2018nonlinear} as a few of examples. However, in our case, $b$ and $K$ are not bounded and $b$ can have polynomial growth at infinity, so the proofs in these works do not quite fit our setting here. For example, in the work of \cite{carrillo2003kinetic,cattiaux2008}, $b=-\nabla V$ and they require $\nabla V\cdot x\ge C$ for some constant while we allow $b\cdot x\le \beta|x|^2$; also the requirements on the kernel $K(\cdot)$ also do not quite match the setup here. In the work \cite{barbu2018nonlinear}, a certain class of nonlinear Fokker-Planck equations have been studied via the approach of Crandall and Liggett for $m$-accretive operators in $L^1(\R^d)$, but the approach cannot be applied directly to  our case here.
Due to these reasons, we attach a proof of Proposition \ref{pro:nonlinearfpestimates} in Appendix \ref{app:proofofnonlinearFP} for a reference.

In proving the uniqueness of the solution to \eqref{eq:nonFPdensity} in Appendix \ref{app:proofofnonlinearFP}, we have in fact proved the following mean-field limit:
\begin{proposition}
With the same assumptions of Proposition \ref{pro:nonlinearfpestimates}, one has
\begin{gather}
\sup_{0\le t\le T}W_2(\varrho, \mu_N^{(1)})\le \frac{C(T)}{\sqrt{N}},
\end{gather}
where $\mu_N^{(1)}$ is the one marginal distribution of the interacting particle system \eqref{eq:Nparticlesys}. Moreover, if Assumption \ref{ass:kernelstrong} holds, the constant $C(T)$ can be independent of $T$.
\end{proposition}

As long as the existence and uniqueness of the solutions to the nonlinear Fokker-Planck equation have been established, one can regard 
\begin{gather}
\bar{K}(x, t):=\int_{\R^d} K(x-y)\varrho(y,t)\,dy,
\end{gather}
as known, and the properties of $\varrho$ can be studied via the {\it linear} Fokker-Planck equation
\begin{gather}
\partial_t \varrho=-\nabla\cdot[(b(x)+\bar{K}(x, t))\varrho]+\sigma^2\Delta\varrho.
\end{gather}

By the moment estimates of $\varrho$, $\bar{K}(0, t)$ is bounded by the first moment of $\varrho$
and it is Lipschitz continuous with uniform Lipschitz constant $L$.
We consider the time continuity of $\bar{K}$.

\begin{lemma}\label{lmm:increamentrandom}
Under Assumptions \ref{ass:kernelfunctions}, \ref{ass:newrho0}, \ref{ass:polynomialgrowth}, we have for any $\Delta t\in [0, \tau]$,
\begin{gather}
\left|\bar{K}(x,t+\Delta t)-\bar{K}(x, t) \right|
\le C(M_q(\varrho(t)))(1+|x|^q)\tau,
\end{gather}
for some $q>1$, where $M_q(\varrho(t))$ means the $q$-moment of $\varrho$ at $t$. Moreover, if Assumption \ref{ass:kernelfunctions} is replaced by Assumption \ref{ass:kernelstrong}, $C(M_q(\varrho(t)))$ has an upper bound independent of time $t$.
\end{lemma}
\begin{proof}
It can be computed directly that
\[
\begin{split}
\partial_t\bar{K}(x,t)
&=\int_{\R^d} K(x-y)\{-\nabla\cdot[(b(y)+K*\varrho)\varrho]+\sigma^2\Delta_y\varrho\}dy\\
&=-\int_{\R^d} (b(y)+K*\varrho)\varrho\cdot (\nabla K)(x-y)\,dy+\int_{\R^d} \sigma^2(\Delta K)(x-y)\varrho\,dy. 
\end{split}
\]
Since $b$ has polynomial growth and $\nabla K$ is bounded, then
\[
\begin{split}
&\left|-\int_{\R^d} (b(y)+K*\varrho)\varrho\cdot (\nabla K)(x-y)\,dy\right|\\
&\le C\int_{\R^d} (1+|y|^q)\varrho\,dy+C\iint_{\R^d\times\R^d} |K(x-y)|\varrho(x)\varrho(y)\,dxdy.
\end{split}
\]
This is controlled by the moments of $\varrho$.

Moreover, since $\Delta K$ has polynomial growth, 
\[
\begin{split}
\left| \int_{\R^d} \sigma^2(\Delta K)(x-y)\varrho\,dy \right|
&\le \sigma^2 C\int_{\R^d} (1+|x-y|^q)\varrho\,dy \\
&\le C\left(1+\int_{\R^d} (|x|^q+|y|^q)\varrho\,dy\right)\le C(1+|x|^q),
\end{split}
\]
where $C$ depends on the moments of $\varrho$.

Using the results in Proposition \ref{pro:nonlinearfpestimates}, the moments on $[t, t+\Delta t]$ can be controlled by the one at $t$. Since $\tau$ is a fixed small number, we omit the dependence in $\tau$ for the amplification constant,  the claims then follow.
\end{proof}

Before further discussion, we first establish some auxilliaury results regarding the following linear Fokker-Planck equation
\begin{gather}\label{eq:linearfp}
\partial_tf=-\nabla\cdot(b_1(x, t)f)\,dt+\sigma^2\Delta f=:\cL_{b_1}^*(f).
\end{gather}
We will assume $b_1(x,t)$ satisfies
\begin{gather}\label{eq:b1cond1}
(x-y)\cdot(b_1(x,t)-b(y,t))\le \beta_1|x-y|^2.
\end{gather}
We say $b_1$ satisfies the strong confinement condition if
$\beta_1<0$. We also denote $S_{s,t}$ the solution operator from time $s$ to time $t$: 
\begin{gather}
f_t=:S_{s,t}f_s.
\end{gather}

There are many classical results on the parabolic equation \eqref{eq:linearfp} with bounded drifts $b_1$ or drifts with linear growth (see, for example, \cite{ladyvzenskaja1988linear}). However, the results for drifts with polynomial growth seem limited.  Below, we will show some results, especially the properties of the fundamental solutions, for drifts with polynomial growth (see Lemma \ref{lmm:fundamentalmoments} and Proposition \ref{pro:derivativemoments}) to fulfill our needs.

\begin{lemma}\label{lmm:momentevol}
Consider equation \eqref{eq:linearfp}, where $b_1$ satisfies \eqref{eq:b1cond1}. Also, assume the derivatives of $b_1(x,t)$ have polynomial growth and $\sup_{t\ge0 }|b(0, t)|<\infty$.   
Then, for $q\ge 1$,
\begin{enumerate}[(i)]
\item  For any $g \in L^1(\R^n)$, one has
\[
\sup_{ \Delta t \le T}\int_{\R^d}(1+ |x|^q) |S_{t, t+\Delta t}g|\,dx\le C(T)\int_{\R^d} (1+|x|^q)|g(x)|\,dx.
\]

\item 
If $b_1$ satisfies the strong confinement condition $\beta_1<0$, $C(T)$ in item (i) can be made independent of $T$. Moreover, when $\sigma>0$ and $\int_{\R^d} g\,dx=0$, $\beta_1<0$ implies that
\begin{gather*}
\int_{\R^d}  (1+|x|^q) |S_{t,t+\Delta t}g|\,dx\le P(M_{q_1}(|g|))e^{-\delta \Delta t},
\end{gather*}
where $\delta>0$ is independent of $g$, $P(\cdot)$ is some polynomial,  $q_1>q$ is some suitable number, and $M_{q_1}(|g|)$ means the $q_1$-moment of $|g|$.

\item In the case $b_1$ does not depend on time so that $S_{s,t}=e^{(t-s)\cL_{b_1}^*}$, one also has
\[
\sup_{\Delta t\le T}\int_{\R^d}(1+|x|^q) |(\cL_{b_1}^*)^m S_{t,t+\Delta t}g|\,dx\le C(T)\int_{\R^d} (1+|x|^q)|(\cL_{b_1}^*)^m g|\,dx.
\]
\end{enumerate}
\end{lemma}

\begin{proof}
For (i), one decomposes $g=: g^+-g^-$ where $g^+=\max(g, 0)$ and $g^-=-\min(g, 0)$. 
Then, $S_{t, t+\Delta t}g=(S_{t, t+\Delta t}g^+)-(S_{t, t+\Delta t}g^-)$ with each 
of them being nonnegative. The operator $S_{t,t+\Delta t}$ is $L^1$-contraction, so we focus on the $q$-moments only. Following similar approaches of Step 1 in Appendix \ref{app:proofofnonlinearFP}, one can show that the moments of $S_{t,t+\Delta t}g^{\pm}$ can be controlled by those of $g^{\pm}$. Hence, the moments of $S_{t,t+\Delta t}g$ have the desired estimates. We skip the details.

Regarding (ii), we first note that the $q$ moments
of $S_{t, t_1}g$ can be uniformly controlled by moments of $|g|$, due to similar reasons. Then, one can consider the measures  $\mu^{\pm}:=\frac{1}{\|g^{\pm}\|_{L^1}}S_{t,t+\Delta t}g^{\pm}$. Using standard techniques of Markov chains (see \cite[Appendix A]{mattingly2020ergodicity} and \cite[Chapters 15-16]{MR2509253} ), one can show that
\[
\|\mu^+-\mu^-\|_{TV}\le P_1(M_{q_1}(|\mu|))e^{-\delta' \Delta t}
\Rightarrow \|S_{t,t+\Delta t}g\|_{TV}\le P(M_{q_1}(|g|))e^{-\delta' \Delta t},
\]
for some $q_1>q$ and polynomials $P_1(\cdot)$, $P(\cdot)$.
Then 
\[
\begin{split}
\int_{\R^d} |x|^q |S_{t, t+\Delta t}g|(dx)
&= \frac{1}{2}\|g\|_{L^1}\int |x|^{q} |\mu^+-\mu^-|(dx) \\
&\le C\|g\|_{L^1}\sqrt{\int |x|^{2q}|\mu^+-\mu^-|(dx)}\sqrt{\|\mu^+-\mu^-\|_{TV}}\\
&= C\sqrt{\int_{\R^d} |x|^{2q}|S_{t,t+\Delta t}g|(dx)}\sqrt{\|S_{t,t+\Delta t}g\|_{TV}}.
\end{split}
\]
Further, due to 
\[
\sqrt{\int_{\R^d} |x|^{2q}|S_{t,t+\Delta t}g|(dx)}\le \frac{1}{2}(1+M_{2q}(|g|)),
\]
one can then choose another $q_1$ large enough such that claims in (ii) hold.

For (iii), we just note that
\[
\partial_t((\cL_{b_1}^*)^mf)=\cL_{b_1}^*((\cL_{b_1}^*)^mf).
\]
Then, we apply the property of $e^{t\cL_{b_1}^*}$ proved in the first part (i).
\end{proof}

\begin{remark}
For (ii), if $\sigma=0$, even if the strong confinement condition is satisfied, 
$\|\mu^+-\mu^-\|_{TV}$ may not decay. However, we believe that when $b_1(x,t)\to b_{\infty}(x)$, then 
\[
\int_{\R^d} |x-x_*|^q |S_{t,t+\Delta t}g|\,dx\le C(M_{q_1}(|g|))e^{-\delta \Delta t}
\]
still holds for the limiting point $x_*$ of the trajectories. We do not explore this in this work.
\end{remark}

It is well-known that the linear equation \eqref{eq:linearfp} has a transition density
$\Phi(x, t; y, s)$ solving \eqref{eq:linearfp} for $t>s$ with initial data
$\Phi(x, s; y, s)=\delta(x-y)$.
Then,
\begin{gather}
(S_{s,t}g)(x)=\int_{\R^d} \Phi(x, t; y, s)g(y)\,dy.
\end{gather}
Hence, the property of $\Phi$ is important.
\begin{lemma}\label{lmm:fundamentalmoments}
Consider equation \eqref{eq:linearfp} with $\sigma>0$, and $b_1$ satisfying \eqref{eq:b1cond1}. Also, assume the derivatives of $b_1(x,t)$ have polynomial growth and $\sup_{t\ge0 }|b_1(0, t)|<\infty$.  Then, for all $0\le s<t\le T$, we have
\begin{gather}
\int_{\R^d}(1+|x|^q)|\nabla_y \Phi(x,t; y,s)|\,dx
\le C(T)P(|y|)(1+(t-s)^{-1/2}),
\end{gather}
for some polynomial $P(\cdot)$. If $\beta_1<0$, 
\begin{gather}
\int_{\R^d}(1+|x|^q)|\nabla_y \Phi(x,t; y,s)|\,dx
\le CP(|y|)(1+(t-s)^{-1/2})e^{-\delta(t-s)}
\end{gather}
for some $\delta>0$.
\end{lemma}
Proof of Lemma \ref{lmm:fundamentalmoments} is tedious, and we defer it 
to Appendix \ref{app:fundmoment}. Below, we aim to consider the moments of the derivatives of $\varrho$. 
Now, we recall the standard multi-index notation used in PDE community:
\begin{gather}\label{eq:multind}
D^{\alpha}:=\prod_{j=1}^d\partial_{j}^{\alpha^j},~~\alpha=(\alpha^1,\cdots, \alpha^d).
\end{gather}
The length of the index is defined as
$|\alpha|:=\sum_{j=1}^d \alpha^j$.

The following proposition is helpful for our estimates later.
\begin{proposition}\label{pro:derivativemoments}
Let Assumptions \ref{ass:kernelfunctions}, \ref{ass:newrho0}, and \ref{ass:polynomialgrowth} hold. Then, for any multi-index $\alpha$, it holds that
\begin{gather}
\sup_{t\le T}\int_{\R^d} (1+|x|^q) |D^{\alpha}\varrho|dx\le C(\alpha, q, T).
\end{gather}
If $\sigma>0$ and Assumption \ref{ass:kernelstrong} holds, then 
\begin{gather}
\sup_{t>0}\int_{\R^d} (1+|x|^q) |D^{\alpha}\varrho|dx<\infty.
\end{gather}
\end{proposition}

\begin{proof}
We set
\[
b_1(x,t):=b(x)+\bar{K}(x,t),
\]
which is regarded as known (since existence and uniqueness of $\varrho$ have been established).

In the case $\sigma=0$, consider the characteristics satisfying
\[
\dot{Z}=b(Z),~Z(0; y)=y.
\]
Using the one-sided Lipschitz condition in Assumption \ref{ass:kernelfunctions}, one has $v\cdot\nabla b_1(x,t)\cdot v\le \beta_1|v|^2$ for any $v, x$ with $\beta_1=\beta+2L$.
 With this and induction, one can show that
$|\partial_{y_i}Z|\le Ce^{\beta_1 t}$ and $D_y^{\alpha}Z(t; y)$ is controlled by polynomials of $|y|$ for higher order $\alpha$. Using $\varrho=Z_{\#}\varrho_0$, the claim can be proved. We omit the details.

Now, we focus on $\sigma>0$.  We do by induction on the derivatives of $\varrho$.
Let $\ell=|\alpha|$. We know already that the claim holds for $\ell=0$.

Suppose the claim is true for $\ell-1$ with $\ell\ge 1$. Now, we consider $\ell$. 
One can see that
\begin{gather*}
\partial_t D^{\alpha}\varrho=-\nabla\cdot(b_1(x,t) D^{\alpha}\varrho)+\sigma^2\Delta D^{\alpha}\varrho
+\sum_{|\beta|\le \ell-1} C_{\beta}\nabla\cdot[ f_{\beta}(x) D^{\beta}\varrho].
\end{gather*}
Here, $f_{\beta}$ are some functions with polynomial growth.
Then, we have
\[
D^{\alpha}\varrho=S_{0,t}D^{\alpha}\varrho_0
-\int_0^t \sum_{|\beta|\le \ell-1}C_{\beta} \int_{\R^d}\nabla_y\Phi(x,t;y,s)\cdot(f_{\beta}(y) D^{\beta}\varrho(y,s)) \,dy ds.
\]
The claim follows by a direct application of the induction assumption and Lemma \ref{lmm:fundamentalmoments} with $\beta_1=\beta+2L$ or $\beta_1=-r+2L$.
\end{proof}

\subsection{Approximation of the limiting dynamics to the nonlinear Fokker-Planck equation}

To get a feeling how close the dynamics given by $\mathcal{G}_{\infty}$ (the mean field limit of RBM) is to the nonlinear Fokker-Planck equation \eqref{eq:nonlinearFP}, we consider \eqref{eq:firstalgorithm}.
Recall that $\rho^{(p)}(\cdots, t_k)=\tilde{\mu}(\cdot, t_k)^{\otimes p}$, with order $\tau$ error,  \eqref{eq:firstalgorithm} is approximated as
\begin{multline}
\partial_t\rho^{(p)}
=-\sum_{i=1}^p\nabla_{x_i}\cdot\left(\left[b(x_i)+\frac{1}{p-1}\sum_{j: j\neq i}K(x_i-x_j)\right]\prod_{j=1}^p\tilde{\mu}(x_j,t_k)\right)\\
+\sigma^2\sum_{i=1}^p \Delta_{x_i}\rho^{(p)}+O(\tau).
\end{multline}
Since we are curious about how the marginal distribution is evolving, one may take the integrals on $x_2,\cdots, x_p$ and have:
\[
\partial_t\tilde{\rho}
=-\nabla_{x_1}\cdot([b(x_1)+K*\tilde{\mu}(\cdot,t_k)]\tilde{\mu}(x_1, t_k))
+\sigma^2\Delta_{x_1}\tilde{\rho}+O(\tau).
\]
Since $\tilde{\rho}:=\int \rho^{(p)}\,dx_2\cdots dx_p$ is equal to $\tilde{\mu}(\cdot, t_k)$
initially, one finds that this is close to \eqref{eq:nonlinearFP} already. Thus, one expects that the overall error between $\mathcal{G}_{\infty}^k(\varrho_0)$ and $\varrho(k\tau)$ is like $O(\tau)$.

We now state the main results in this section.
\begin{theorem}\label{thm:w2distance}
Let $\varrho$ be the solution to the nonlinear Fokker-Planck equation \eqref{eq:nonFPdensity}.
Suppose Assumptions \ref{ass:kernelfunctions}, \ref{ass:newrho0} and \ref{ass:polynomialgrowth} hold. Then, 
\begin{gather}
\sup_{n: n\tau\le T}W_{1}(\mathcal{G}_{\infty}^n(\varrho_0), \varrho(n\tau))\le C(T)\tau.
\end{gather}

If Assumption \ref{ass:kernelstrong} is assumed in place of Assumption \ref{ass:kernelfunctions} and also $\sigma>0$, then
\begin{gather}
\sup_{n\ge 0}W_{1}(\mathcal{G}_{\infty}^n(\varrho_0), \varrho(n\tau))\le C\tau.
\end{gather}
Consequently, the invariant measures (see Proposition \ref{pro:stabilityW2} and Proposition \ref{pro:nonlinearfpestimates} for the related notations) satisfy
\begin{gather}
W_{1}(\pi_{\tau}, \pi)\le C\tau.
\end{gather}
\end{theorem}

Below, we aim to prove Theorem \ref{thm:w2distance}. We first establish the one-step error and then give the global estimate.

Define
\begin{gather}\label{eq:derimoment}
M_{q,\ell }^{(k)} :=\sum_{|\alpha|\le \ell }\int_{\R^d}(1+|x|^q) |D^{\alpha}\varrho(x, t_k))|\,dx,
\end{gather}
which is is the moments of $|D^{\alpha}\varrho(\cdot, t_k)|$ for $|\alpha| \le \ell$ (see \eqref{eq:multind} for the multi-index notation).
In fact, we have the following result provided that $\varrho$ is smooth enough.
\begin{lemma}\label{lmm:onesteptononlinearFP}
Suppose Assumptions \ref{ass:kernelfunctions} and \ref{ass:polynomialgrowth} hold. Let $t_k \le T-\tau$.
Then,
\[
W_1(\mathcal{G}_{\infty}(\varrho(\cdot, t_k)), \varrho(\cdot, t_{k+1}))\le g(M_{q,4}^{(k)})\tau^2,
\]
for some $q>1$ and nondecreasing function $g(\cdot)$, where $M_{q,4}^{(k)}$ is defined in \eqref{eq:derimoment}.
\end{lemma}

\begin{proof}

For the notational convenience in this proof, we denote, only in this proof,
\[
\varrho_k(\cdot)\equiv \varrho(\cdot, t_k).
\]

{\bf Step 1--}   Consider the SDE corresponding to the nonlinear Fokker-Planck equation \eqref{eq:nonFPdensity}:
\begin{gather*}
dX=[b(X)+K(\cdot)*\varrho(\cdot,t)(X)]\,dt+\sqrt{2}\sigma dW.
\end{gather*}
Denote $\bar{K}(X):=\int_{\R^d} K(X-z)\varrho_k(z)\,dz$,
then we have
\begin{gather}
dX=[b(X)+\bar{K}(X)+R]\,dt+\sqrt{2}\sigma\, dW,
\end{gather}
where, by a similary calculation as in the proof of Lemma \ref{lmm:increamentrandom},
\[
|R|\le C(M_{q_1,0})(1+|X(t_k)|^q)\tau,
\]
for some $q_1>1$. In fact, $C$ depends on the moments of $\varrho(\cdot, t)$
for $t\in [t_k, t_{k+1}]$, which can be controlled by the ones at $t_k$.

We show that the law of $X$ is close in $W_1$ to the law generated by the following SDE:
\begin{gather}\label{eq:meanfieldSDE1}
d\hat{X}=[b(\hat{X})+\bar{K}(\hat{X})]\,dt+\sqrt{2}\sigma\, dW.
\end{gather}
To do this, we estimate $\E|X-\hat{X}|$ under the synchronization coupling (i.e., using the same Brownian motion).  In fact, 
\[
\frac{d}{dt}\E|X-\hat{X}|\le C\E|X-\hat{X}|+C\E|R|.
\]
Clearly, $\E|R|\le C(M_{q_2,0})\tau$ for some $q_2>1$.

Denote (recall that $\mathscr{L}$ means the law of a random variable)
\begin{gather}
\tilde{\mathcal{S}}(\varrho_k):=\mathscr{L}(\hat{X}(\tau)).
\end{gather}
Then, applying Gr\"onwall's lemma yields
\[
W_1(\varrho(\cdot, t_{k+1}), \tilde{\mathcal{S}}(\varrho_k))\le C(M_{q,0})\tau^2.
\]

\vskip 0.2 in

{\bf Step 2--} Compare $\tilde{\mathcal{S}}(\varrho_k)$ with $\mathcal{G}_{\infty}(\varrho_k)$.

We compare the law of $\hat{X}$ in \eqref{eq:meanfieldSDE1} (i.e., $\tilde{\mathcal{S}}(\varrho_k)$) with the law of $Y^1$ (i.e., $\mathcal{G}_{\infty}(\varrho_k)$) given by 
\begin{gather}
dY^i=b(Y^i)\,dt+\frac{1}{p-1}\sum_{j=1,j\neq i}^{p}K(Y^i-Y^j)\,dt
+\sqrt{2}\sigma\,dW^i,~~i=1,\cdots, p,
\end{gather}
with the initial data drawn from $\varrho_k^{\otimes p}$.
The main strategy is to use Lemma \ref{lmm:W2byTV}, so we need to estimate the 
difference of these two distributions and control the moments of this difference.

It is clear that $\tilde{\mathcal{S}}(\varrho_k)=e^{\tau \hat{\cL}^*}\varrho_k$,
where $\hat{\cL}^*$ is given by (for $\rho$ in its domain)
\begin{gather}
\begin{split}
\hat{\cL}^*(\rho)(x)&:=-\nabla\cdot\left(\left[b(x)+\int_{\R^d} K(x-x_2)\varrho_k(x_2)dx_2\right] \rho(x)\right)+\sigma^2\Delta_x\rho(x)\\
&=-\int_{\R^d} dx_2 \varrho_k(x_2)[\nabla\cdot([b(x)+K(x-x_2)]\rho(x))+\sigma^2\Delta_x\rho(x)].
\end{split}
\end{gather}

Denote the Fokker-Planck operator for the evolution of $(Y^1, \cdots, Y^p)$ by
\[
\bar{\mathcal{L}}^*:=-\sum_{i=1}^p\nabla_{x_i}\cdot([b(x_i)
+\frac{1}{p-1}\sum_{j: j\neq i}K(x_i-x_j)] \cdot)
+\sigma^2\sum_{i=1}^p \Delta_{x_i}.
\]
Then, the law of $Y^1$ at $\tau$ is given by
\begin{gather}
\mathcal{G}_{\infty}(\varrho_k)=\int_{(\R^d)^{p-1}} e^{\tau \bar{\mathcal{L}}^*}\prod_{i=1}^p\varrho_k(x_i)\,dx_2\cdots dx_p.
\end{gather}

First note 
\begin{gather}
\tilde{S}(\varrho_k)(x)=\varrho_k(x)+\tau\hat{\cL}^*\varrho_k(x)
+\int_0^{\tau}(\tau-s)(\hat{\cL}^*)^2
e^{s\hat{\cL}^*}\varrho_k\,ds,
\end{gather}
while
\begin{gather}\label{eq:localW2G}
\begin{split}
\mathcal{G}_{\infty}(\varrho_k)(x_1)&=\int_{(\R^d)^{p-1}} \prod_{i=1}^p\varrho_k(x_i)\,dx_2\cdots dx_p
+\tau \int_{(\R^d)^{p-1}} \bar{\mathcal{L}}^*\prod_{i=1}^p\varrho_k(x_i)\,dx_2\cdots dx_p\\
&+\int_0^{\tau}(\tau-s)\int_{(\R^d)^{p-1}} (\bar{\mathcal{L}}^*)^2 e^{s\bar{\mathcal{L}}^*}
\prod_{i=1}^p\varrho_k(x_i)\,dx_2\cdots dx_p\, ds.
\end{split}
\end{gather}
The first line of \eqref{eq:localW2G} is reduced to
\begin{multline}\label{eq:operatorconsistency}
\varrho_k(x_1)-\tau \int_{(\R^d)^{p-1}} \nabla_{x_1}\cdot\left( \left[b(x_1)+
\frac{1}{p-1}\sum_{j: j\neq 1}K(x_1-x_j) \right] \prod_{i=1}^p\varrho_k(x_i) \right) dx_2\cdots dx_p\\
+\tau \sigma^2\Delta_{x_1}\varrho_k(x_1)=\varrho_k(x_1)+\tau \hat{\cL}^*\varrho_k(x_1),
\end{multline}
where we used
\begin{multline*}
\int_{(\R^d)^{p-1}} \nabla_{x_1}\cdot\left( \left[b(x_1)+
\frac{1}{p-1}\sum_{j: j\neq 1}K(x_1-x_j) \right] \prod_{i=1}^p\varrho_k(x_i) \right) dx_2\cdots dx_p\\
= \nabla_{x_1}\cdot\left( \left[b(x_1)+
\int_{\R^d} K(x_1-y)\varrho_k(y)dy \right] \varrho_k(x_1) \right).
\end{multline*}

Hence, we find
\begin{multline}
|\tilde{S}(\varrho_k)(x)-\mathcal{G}_{\infty}(\varrho_k)(x)|\le \\
\int_0^{\tau}(\tau-s)\left[|(\hat{\cL}^*)^2
e^{s\hat{\cL}^*}\varrho_k|+
\left|\int_{(\R^d)^{p-1}} (\bar{\mathcal{L}}^*)^2 e^{s\bar{\mathcal{L}}^*}
\Big(\prod_{i=1}^p\varrho_k(x_i)\Big)\,dx_2\cdots dx_p \right| \right]\,ds.
\end{multline}

Now, we will apply Lemma \ref{lmm:W2byTV} for $\sq=1$ with
$\delta=\tau^2$ and $\hat{\mu}=\hat{\rho}\,dx$ with
\[
\hat{\rho}=\frac{1}{\tau^2}\int_0^{\tau}(\tau-s)\left[|(\hat{\cL}^*)^2
e^{s\hat{\cL}^*}\varrho_k|+\left|\int_{(\R^d)^{p-1}} (\bar{\mathcal{L}}^*)^2 e^{s\bar{\mathcal{L}}^*}
\Big(\prod_{i=1}^p\varrho_k(x_i)\Big)\,dx_2\cdots dx_p \right| \right]\,ds.
\]
The moment $M_1$ of $\hat{\mu}$ is controlled by $C(M_{q,4})$ for a constnat $C$ depending on $M_{q,4}$. To see this, we first remark that for $x\in \R^d$, one has
$1+|x|\le 2+|x|^2$. 
Both $\hat{\cL}^*$ and $\bar{\cL}^*$ are constant operators, and then one has by Lemma \ref{lmm:momentevol} (iii) that for some $q>1$,
\[
\int_{\R^d} (2+|x|^2)\hat{\rho}\, dx \le  C M_{q,4}.
\]
To illustrate how this is estimated, we take the second term as an example:
\[
\begin{split}
&\int_{\R^d} (2+|x|^2)\left|\int_{(\R^d)^{p-1}} (\bar{\mathcal{L}}^*)^2 e^{s\bar{\mathcal{L}}^*}
\prod_{i=1}^p\varrho_k(x_i)\,dx_2\cdots dx_p \right| \,dx \\
&\le \int_{(\R^d)^p} (2+|x_1|^2) \left|(\bar{\mathcal{L}}^*)^2 e^{s\bar{\mathcal{L}}^*}
\prod_{i=1}^p\varrho_k(x_i) \right|\,dx_1\cdots dx_p\\
&=\int_{(\R^d)^p} (2+\frac{1}{p}\sum_i |x_i|^2) \left|(\bar{\mathcal{L}}^*)^2 e^{s\bar{\mathcal{L}}^*}
\prod_{i=1}^p\varrho_k(x_i) \right|\,dx_1\cdots dx_p\\
&\le  \int_{(\R^d)^p} (2+\frac{1}{p}\sum_i |x_i|^2) \left|(\bar{\mathcal{L}}^*)^2 
\prod_{i=1}^p\varrho_k(x_i)\right|\,dx_1\cdots dx_p.
\end{split}
\]
This is controlled by $M_{q,4}$. Note that the dependence in $\tau$ for the constant $C(\tau)$ in Lemma \ref{lmm:momentevol} has been omitted since $\tau\lesssim O(1)$.

Lastly, the constants $C(M_{q,0})$ and $C(M_{q,4})$ clearly have an upper bound
$g(M_{q,4})$ with $g$ nondecreasing, defined on $[0,\infty)$.
\end{proof}

With the key one-step estimate established in Lemma \ref{lmm:onesteptononlinearFP} above, we can now finish the proof of Theorem \ref{thm:w2distance}.

\begin{proof}[Proof of Theorem \ref{thm:w2distance}]

By the semigroup property of $\{\mathcal{G}_{\infty}^k\}$, one can find easily that
\[
W_1(\mathcal{G}_{\infty}^n(\varrho_0), \varrho(n\tau))
\le \sum_{m=1}^nW_1\Big(\mathcal{G}_{\infty}^{n-m+1}(\varrho((m-1)\tau))
, \mathcal{G}_{\infty}^{n-m}(\varrho(m\tau)) \Big).
\]

By Proposition \ref{pro:stabilityW2}, under Assumption \ref{ass:kernelfunctions} and Assumption \ref{ass:polynomialgrowth}, one has for $n\tau\le T$ that
\begin{multline*}
\sum_{m=1}^{n}W_1\Big(\mathcal{G}_{\infty}^{n-m+1}(\varrho((m-1)\tau))
, \mathcal{G}_{\infty}^{n-m}(\varrho(m\tau)) \Big) \\
\le \sum_{m=1}^n e^{(\beta+2L)(n-m)\tau}W_1(\mathcal{G}_{\infty}(\varrho((m-1)\tau)), \varrho(m\tau) ).
\end{multline*}
Combining Proposition \ref{pro:derivativemoments} and Lemma \ref{lmm:onesteptononlinearFP}, $W_1(\mathcal{G}_{\infty}(\varrho((m-1)\tau)), \varrho(m\tau) )\le C(T)\tau^2$ and thus the claim follows.

Under Assumption \ref{ass:kernelstrong} and Assumption \ref{ass:polynomialgrowth}, using Proposition \ref{pro:stabilityW2} and Proposition \ref{pro:derivativemoments}, the above estimates can be changed by replacing $\alpha$ with $-(r-2L)$, and $W_1(\mathcal{G}_{\infty}(\varrho((m-1)\tau)), \varrho(m\tau) )$ now is bounded by $C\tau^2$ with $C$ uniform in $T$. Hence, the conclusions follow easily.
\end{proof}

\section{Some helpful discussions}\label{sec:discussion}

In this section, we perform some helpful discussions to deepen the understanding and extend the results to second order interacting particle systems.

\subsection{The mean field limit for \texorpdfstring{$\tau\ll 1$}{Lg}}\label{subsec:limitcommute}

Formally, as $\tau\to 0$, the equation for $Y^1$ in \eqref{eq:RBMmeanfieldSDE} tends to (i.e., the limit for $\lim_{\tau\to 0}\lim_{N\to\infty}$) the SDE
\begin{gather}
dY=b(Y)\,dt+\frac{1}{p-1}\sum_{j=1}^{p-1}K(Y-Y_j)\,dt+\sqrt{2}\sigma\, dW,
\end{gather}
with $Y_j\sim \mathscr{L}(Y)$ being i.i.d., and  $\{Y_j(s_i)\}$'s are independent for different time points $s_i$.
Theorem \ref{thm:w2distance} essentially tells us that the law of this SDE obeys the same nonlinear Fokker-Planck equation \eqref{eq:nonlinearFP}, which was satisfied by the law of the following seemingly different SDE
\begin{gather}
dX=b(X)\,dt+\left(\int_{\R^d} K(X-y)\varrho(y,t)\,dy \right)\,dt+\sqrt{2}\sigma\,dW,
~~~\varrho(x, t)\,dx=\mathscr{L}(X(t)).
\end{gather}
See Fig. \ref{fig:sde} for illustration (compare with Fig. \ref{fig:operator}).

To understand this, we consider a small but fixed $\tau$, and the following SDEs (with the force field frozen at $t_k$):
\begin{gather}\label{eq:SDEfrozen}
\begin{split}
&d\hat{Y}=b(Y)\,dt+\frac{1}{p-1}\sum_{j=1}^{p-1}K(\hat{Y}-Y_{0}^j)\,dt+\sqrt{2}\sigma dW,~~~Y_{0}^j\sim \mathscr{L}(Y(t_k)),\\
&d\hat{X}=b(\hat{X})\,dt+\left(\int_{\R^d} K(\hat{X}-y)\varrho(y,t_k)\,dy\right)\,dt+\sqrt{2}\sigma\,dW.
\end{split}
\end{gather}
The probability density for the former at $t_k+\tau$ is  $\int_{\R^d} dy \varrho(y, t_k)e^{\tau\cL_y^*}\varrho(\cdot, t_k)$,  where
\[
\cL_y^*=-\nabla\cdot\left([b(x)+K(x-y)] \cdot\right)+\sigma^2\Delta_x,
\]
 while the probability density for the latter is  $ e^{\tau\hat{\cL}^*}\varrho(\cdot,t_k)$
with
\[
\hat{\cL}^* =-\nabla\cdot\left( \left[b(x)+\int_{\R^d} K(x-x_2)\varrho(x_2, t_k)dx_2\right] \cdot\right)+\sigma^2\Delta_x=\int_{\R^d} dy \varrho(y, t_k) \cL_y^*.
\]
Clearly, to the leading order, the changing rates of the probability densities are the same.

\begin{figure}
\begin{center}
	\includegraphics[width=\textwidth]{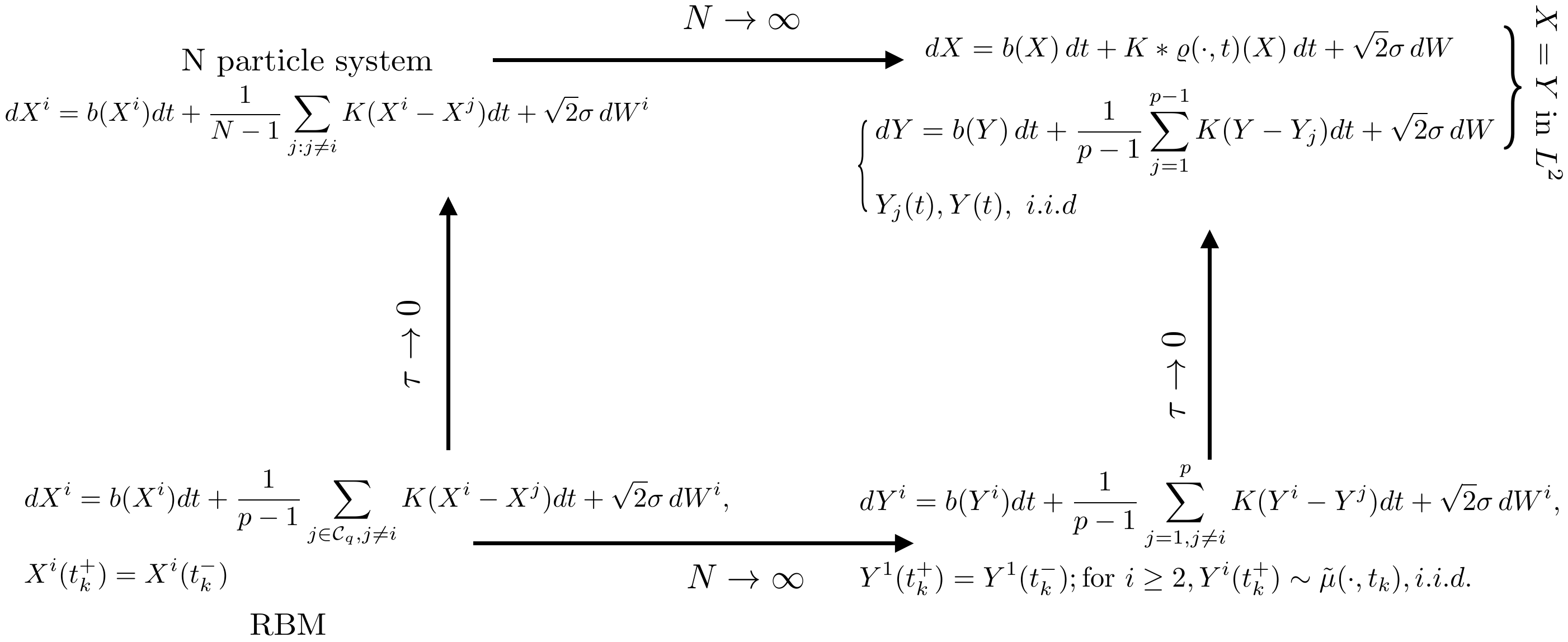}
\end{center}
\caption{Illustration of the various SDEs in different regime.}
\label{fig:sde}
\end{figure}

In  Fig. \ref{fig:sde} we have made a stronger claim that the $X$ and $Y$ processes in the right-upper corner are equal in $L^2$, instead of ``equal in law", if the Brownian motions $W$ used are the same. To see this, one may compute
\[
\frac{d}{dt}\E |X-Y|^2=2\E(X-Y)\cdot(b(X)-b(Y))
+2\E(X-Y)\cdot(K*\varrho(\cdot, t)(X)-\frac{1}{p-1}\sum_{j=1}^{p-1}K(Y-Y_j)).
\]
Since $Y_j(t)$ is independent of $Y(t)$ and $X(t)$, one has
\[
\E(X-Y)\cdot(K*\varrho(\cdot, t)(X)-\frac{1}{p-1}\sum_{j=1}^{p-1}K(Y-Y_j))
=\E(X-Y)\cdot(K*\varrho(\cdot, t)(X)-K*\bar{\varrho}(\cdot,t)(Y)),
\]
where $\bar{\varrho}$ is the law of $Y$. Taking $\tau\to 0$ in Theorem \ref{thm:w2distance}, $\bar{\varrho}=\varrho$. Hence,
one actually has $\frac{d}{dt}\E|X-Y|^2\le 2(\beta+L)\E|X-Y|^2$. Hence, $X=Y$ in $L^2$.

\subsection{Regarding the approximation in Lemma \ref{lmm:onesteptononlinearFP}}\label{subsec:cmtlemma43}

Usually, the Wasserstein distance (especially $W_2$) was estimated using the SDEs. A natural question is therefore whether one can estimate the Wasserstein distance in Lemma \ref{lmm:onesteptononlinearFP} via the SDE approach. 

Below, we illustrate the issue using the $W_2$ distance and the approximating problem \eqref{eq:SDEfrozen} (with the force expressions frozen). Here, we assume the Brownian motions used are the same. The values $Y_0^j$ are i.i.d., drawn from $\varrho(\cdot)$. 

We compute that
\[
\frac{d}{dt}\E|\hat{X}-\hat{Y}|^2
=\E(\hat{X}-\hat{Y})\cdot(b(\hat{X})-b(\hat{Y}))+D,
\]
where
\[
D=\E(\hat{X}-\hat{Y})\cdot\left(\bar{K}(\hat{X})-\frac{1}{p-1}\sum_{j=1}^{p-1}K(\hat{Y}-Y_0^j)\right).
\]
Clearly, for fixed $x$, 
\begin{gather}\label{eq:randomforceconsistency}
\E\frac{1}{p-1}\sum_{j=1}^{p-1}K(x-Y_0^j)
=\bar{K}(x).
\end{gather}
Hence, if $\hat{Y}$ is independent of $Y_0^j$'s, then this term can be controlled as
\[
\E(\hat{X}-\hat{Y})\cdot(\bar{K}(\hat{X})-\bar{K}(\hat{Y}))
\le C\E|\hat{X}-\hat{Y}|^2.
\]
One is thus tempted to believe that even though that $\hat{Y}$ is not independent of $Y_0^j$,
one can do It\^o-Taylor expansion and the extra term is small enough, which can yields the desired error.

Unfortunately, if one is going to do the It\^o-Taylor expansion in $\hat{Y}$, one may find that
$D=O(\tau)$. In fact,
\[
\begin{split}
&(\hat{X}-\hat{Y})\cdot\left(\bar{K}(\hat{X})-\frac{1}{p-1}\sum_{j=1}^{p-1}K(\hat{Y}-Y_0^j)\right)\\
&=\int_0^t\left(\bar{K}(\hat{X}(s))-\frac{1}{p-1}\sum_{j=1}^{p-1}K(\hat{Y}(s)-Y_0^j)\right)
\cdot \left(\bar{K}(\hat{X}(t))-\frac{1}{p-1}\sum_{j=1}^{p-1}K(\hat{Y}(t)-Y_0^j)\right)\,ds.
\end{split}
\]
If we take expectation, the variance of the random force 
$\frac{1}{p-1}\sum_{j=1}^{p-1}K(x-X_0^j)$ appears, which gives $D=O(\tau)$.
Hence, this estimate is not good and the mean square error is only  like $\sqrt{\E|\hat{X}-\hat{Y}|^2}=O(\tau)$. This means that the consistency \eqref{eq:randomforceconsistency} brings no benefit for this mean square error!

Intrinsically, the mean square error above is roughly comparable to 
\[
\int \varrho(z_1)\cdots\varrho(z_j)W_2^2(e^{\tau\hat{\cL}^*}\varrho, e^{t\cL_{z_1,\cdot,z_j}^*}\varrho)
\,dz_1\cdots dz_j.
\]
What we care about is the distance between $e^{\tau\hat{\cL}^*}\varrho$ and $\int \varrho(z_1)\cdots\varrho(z_j) e^{t\cL_{z_1,\cdot,z_j}^*}\varrho\,dz_1\cdots dz_j$.
The former involves the variance introduced by the random force while the latter does not have this issue and uses the consistency \eqref{eq:randomforceconsistency}. This is why we used the total variation norm to obtain the one-step error under $W_1$ distance in Lemma  \ref{lmm:onesteptononlinearFP}.

\subsection{Approximation using weak convergence}

The weak convergence is another popular gauge of the convergence of probability $\mathcal{G}_{\infty}^k(\varrho_0)$ to $\varrho(k\tau)$ \cite{milstein2013stochastic,kloeden2013numerical}.

Pick a test function $\varphi$, using a consistency condition similar to \eqref{eq:operatorconsistency}, it is not very hard to show
\begin{gather}
\left|\int_{\R^d} \varphi(y) \mathcal{S}(\tau)(\mu)(dy)-\int_{\R^d} \varphi(y) \mathcal{G}_{\infty}(\mu)(dy)\right|
\le C\tau^2,
\end{gather}
for any $\mu$, where we recall $\mathcal{S}(t)$ is the evolution operator for \eqref{eq:nonlinearFP} . Hence, the one-step error is easy to control for weak convergence. 
However, the difficulty is to get a certain stability property of the nonlinear dynamics under the weak topology. That means, if two measures are close in the weak topology at some time, then let them evolve under $\mathcal{G}_{\infty}$
for $k$ times, one needs them to be close. Consider
\[
U^n(x):=\int_{\R^d} \varphi(y) \mathcal{G}_{\infty}^n(\delta(y-x))\,dy.
\]
Unlike the linear case (see \cite{feng2018}), it is hard to write $U^n$ as some operator acting on $U^{n-1}$ due to the nonlinearity of $\mathcal{G}_{\infty}$.
Proving the stability of this nonlinear dynamics under weak topology seems challenging, and this is why we chose the Wasserstein metric.

\subsection{A remark for second order systems}
As shown in \cite{jin2020random}, the Random Batch Method applied equally well to second order systems on finite time interval. Repeating the proof here, one can show that similar mean field limit holds for second order systems when $t\in [0, T]$.
 In particular, let us consider the models for swarming and flocking considered in  \cite{albi2013}
 \begin{gather}\label{eq:flocking}
 \begin{split}
& \dot{x}_i=v_i,\\
 & \dot{v}_i=\frac{1}{N}\sum_{j} H_{\alpha}(x_i, x_j, v_i)(v_j-v_i).
 \end{split}
 \end{gather}
 Here, $H_{\alpha}(\cdot,\cdot, \cdot)$ is some function modeling the interactions between particles. The mean field limit of \eqref{eq:flocking} for $t\in [0, T]$ takes the following form (rigorous justification needs some assumptions on $H_{\alpha}$; see \cite{jabin2020review})
\begin{gather}\label{eq:meanfieldkinetic}
\begin{split}
&\partial_t f+\nabla_x\cdot(v f)+\nabla_v\cdot(\xi(f)f)=0,\\
& \xi(f)=\int_{\R^{2d}} H_{\alpha}(x, y, v)(w-v)f(y, w, t)\,dwdy.
\end{split}
\end{gather}
Albi and Pareschi in \cite{albi2013} developed some stochastic binary interaction algorithms for the dynamics. The symmetric Nanbu algorithm (Algorithm 4.3) is like the Random Batch Method when $p=2$ and the Random Batch Method can be viewed as generalization of this Nanbu algorithm. When applying the Random Batch Method
to the particle system and consider $N\gg 1$, the dynamics is expected to be close to the following limiting dynamics:
\begin{algorithm}[H]
\caption{(Mean Field Dynamics of RBM for flocking dynamics \eqref{eq:flocking})}\label{meanfield2nd}
\begin{algorithmic}[1]
\State From $t_k$ to $t_{k+1}$, the distribution $f_k$ will be transformed into $f_{k+1}=\mathcal{Q}_{\infty}(f_k)$ as follows.

\State Let $f^{(p)}(\cdots, t_k)=f(\cdot,\cdot, t_{k})^{\otimes p}$ be a probability measure on $(\R^{2d})^{ p}\cong \R^{2pd}$.

\State Evolve $f^{(p)}$ by time $\tau$ according to the following:
\begin{gather}\label{eq:rbmkinetic}
\begin{split}
&\partial_t f^{(p)}+\sum_{i=1}^p\nabla_{x_i}\cdot(v_i f^{(p)})+
\sum_{i=1}^p \nabla_{v_i}\cdot(\xi_i f^{(p)})=0,\\
& \xi_i=\frac{1}{p-1}\sum_{j:j\neq i} H_{\alpha}(x_i, x_j, v_i)(v_j-v_i).
\end{split}
\end{gather}

\State Set
\begin{gather}
f_{k+1}=\mathcal{Q}_{\infty}(f_k):=\int_{(\R^{2d})^{(p-1)}}
f^{(p)}(\cdot,dy_2,\cdots,dy_p;\cdot, dv_2,\cdots, dv_p; t_{k+1}^-).
\end{gather}
\end{algorithmic}
\end{algorithm}

We expect that this nonlinear operator will approximate the nonlinear kinetic equation 
\eqref{eq:meanfieldkinetic}.
In this sense, we believe the $N\to\infty$ limit of the \cite[Algorithm 4.3]{albi2013} will be an analogue
of the dynamics $\mathcal{Q}_{\infty}$ given in Algorithm \ref{meanfield2nd}.

\section{Conclusions}\label{sec:conclusion}
We first identified and justified in this work the mean field limit of RBM for fixed step size $\tau$. Then, we showed that this mean field limit is close to that of the $N$ particle system, though the chaos arises differently in these two dynamics. The current argument of the mean field limit relies on the fact that two particles are unlikely to be related in RBM when $N\to\infty$ for finite iterations. Hence, this argument cannot given a uniform in $\tau$ bound for the speed of the mean field limit.
It will be an interesting topic to investigate how mixing and chaos can be created in RBM after two particles in a batch are separated, so that one may obtain a convergence speed independent of $\tau$.

\section*{Acknowledgement}
S. Jin was partially supported by the NSFC grant No. 31571071.  The work of L. Li was partially sponsored by NSFC 11901389, 11971314, and Shanghai Sailing Program 19YF1421300. The authors are grateful to Yuanyuan Feng for detecting a mistake in the first version of the manuscript and to Haitao Wang for discussion on fundamental solutions of parabolic equations with unbounded drifts.

\appendix

\section{Proof of Proposition \ref{pro:nonlinearfpestimates}}\label{app:proofofnonlinearFP}

{\bf Step 1--A priori estimates on moments and entropy}

We first perform a priori estimates on the moments.  Fix $q\ge 2$.
\[
\begin{split}
\partial_t\int_{\R^d}|x|^q\varrho\,dx
=\,&\int_{\R^d} |x|^q\{-\nabla\cdot[(b(x)+K*\varrho)\varrho] \}\,dx
+\int_{\R^d}|x|^q \sigma^2 \Delta\varrho\,dx \\
=\,&\int_{\R^d}q|x|^{q-2}x\cdot b(x)\varrho\,dx 
+\iint_{\R^d\times\R^d}q|x|^{q-2}x\cdot K(x-y)\varrho(x)\varrho(y)\,dxdy\\
&~~~+\sigma^2\int_{\R^d} q(q-2+d)|x|^{q-2}\varrho\,dx=:I_1+I_2+I_3.
\end{split}
\]
For $I_2$, one has
\[
\begin{split}
\iint_{\R^d\times\R^d}q|x|^{q-2}x\cdot K(x-y)\varrho(x)\varrho(y)\,dxdy
\le q\iint_{\R^d\times\R^d}|x|^{q-2}x\cdot K(0)\varrho(x)\varrho(y)\,dxdy\\
+qL\iint_{\R^d\times\R^d}|x|^{q-1}(|x|+|y|)\varrho(x)\varrho(y)\,dxdy.
\end{split}
\]
By Young's inequality, 
\[
q\iint_{\R^d\times\R^d}|x|^{q-2}x\cdot K(0)\varrho(x)\varrho(y)\,dxdy
\le \delta \int_{\R^d} |x|^q\varrho\,dx+C(\delta).
\]
Also, Young's inequality implies that $|x|^{q-1}|y|\le \frac{q-1}{q}|x|^q+\frac{1}{q}|y|^q$. Hence,
\[
I_2\le q(2L+\delta)\int_{\R^d}|x|^q\varrho\,dx+C(\delta).
\]

If $q=2$, $I_3$ is a constant. Otherwise if $q>2$, one can use Young's inequality and
\[
I_3\le \delta \int_{\R^d} |x|^q\varrho\,dx+C(\delta).
\]

For $I_1$, under Assumption \ref{ass:kernelfunctions}, one has
\[
\begin{split}
I_1&=\int_{\R^d}q|x|^{q-2}x\cdot (b(x)-b(0))\varrho\,dx
+\int_{\R^d}q|x|^{q-2}x\cdot b(0)\varrho\,dx \\
&\le \beta q \int_{\R^d} |x|^q\varrho\,dx+C\int_{\R^d} |x|^{q-1}\varrho\,dx.
\end{split}
\]
Hence,
\[
I_1+I_2+I_3
\le q(\beta+2L+\delta)\int_{\R^d}|x|^q\varrho\,dx+C(\delta),
\]
where the concrete meaning of $\delta$ and $C(\delta)$ have changed.
Using Gr\"onwall inequality, the moments can be controlled. 

Now, we perform a priori estimates on the entropy.
Multiply $1+\log\varrho$ on both sides and integrate:
\[
\frac{d}{dt}\int_{\R^d}\varrho\log\varrho\,dx=-\int_{\R^d} \varrho(x)\nabla\cdot(b(x)+(K*\varrho)(x))\,dx
-4\sigma^2\int_{\R^d} |\nabla\sqrt{\varrho}|^2\,dx.
\]
By the moment control, the first term is bounded on $[0, T]$. Hence, the entropy  can be controlled.

As a remark, in the case $\sigma=0$, $\varrho$ could be zero at some points. In this case $1+\log\varrho$ is not a good test function. This issue will be explained further in Step 2.

\vskip 0.1 in
{\bf Step 2--Existence in $L^{\infty}(0, T; L^1(\R^d))\cap C([0, T]; \bfP(\R^d))$}

Take a smooth function $\chi\in C_c[0,\infty)$ that is $1$ in $[0, 1]$ and zero on $[2,\infty)$.
Consider the following approximating equation
\[
\begin{split}
&\partial_t\rho_N=-\nabla\cdot(b(x)\chi(x/N)\rho_N)
-\nabla\cdot(\rho_N(K*\rho_N))+\Delta\rho_N,\\
&\varrho_N|_{t=0}=\varrho_0.
\end{split}
\]
Now, $b(x)\chi(x/N)$ and $K$ are Lipschitz functions
and $b(x)\chi(x/N)$ is bounded (compactly supported).
The existence of a smooth solution is clear (see, for example,  Appendix A in \cite{carrillo2003kinetic}).
Performing similar estimates as in Step 1, we have
\[
\sup_{N}\sup_{0\le t\le T}\int_{\R^d} |x|^2\varrho_N\,dx\le C(T)
\]
and
\[
\sup_N\sup_{0\le t\le T}\int_{\R^d} \varrho_N\log\varrho_N\,dx\le C(T).
\]
Note that for the entropy, the zeros of $\varrho_N$ may make $1+\log(\varrho_N)$ an invalid test function. We instead multiply
\[
\frac{\varrho_N}{\varrho_N+\e}+\log(\varrho_N+\e)
\]
as the test function for $\e>0$. Then, the left hand side becomes $\frac{d}{dt}\int\varrho_N\log(\varrho_N+\e)\,dx$ (note that $\e\to \varrho\log(\varrho+\e)$ is non-decreasing so later one can take $\e\to 0$ to get desired entropy control). For the right hand side, we note
\[
\nabla\left[\frac{\varrho_N}{\varrho_N+\e}+\log(\varrho_N+\e)\right]
=\frac{(\varrho_N+2\e)\nabla\varrho_N}{(\varrho_N+\e)^2}.
\]
For the transport term, 
\[
b(x)\chi(\frac{x}{N})\frac{\varrho_N(\varrho_N+2\e)\nabla\varrho_N}{(\varrho_N+\e)^2}
=(b(x)\chi(x/N))\cdot\nabla\varrho_N+\e^2(b(x)\chi(x/N))\cdot\nabla\left(\frac{1}{\varrho_N+\e}\right).
\]
Doing integration by parts and sending $\e\to 0$ first, the second term here will vanish.
Through this way, a prior estimate on the entropy can be justified for this approximating sequence.

The moment estimates imply that $\{\varrho_N\,dx\}$ is tight while the entropy estimates imply that $\{\varrho_N\}$ is uniformly integrable. By Dunford-Pettis theorem, $\varrho_N$ converges weakly to some $\varrho\in L_{loc}^1([0, T]\times \R^d)$ and $\varrho\,dx\in C([0, T]; \bfP(\R^d))$.
Moreover, with the moment control and the uniform integrability
\[
\int_{\R^d} K(x-y)\varrho_N(y)\,dy\to \int_{\R^d} K(x-y)\varrho(y)\,dy
\]
pointwise and actually uniformly on compact sets.
With this, then one can easily verify that $\varrho$ is a desired weak solution, with the corresponding moment control.
This will further imply that $\varrho\in L^{\infty}([0, T]; L^1(\R^d))$.

\vskip 0.2 in

{\bf Step 3--Uniqueness and smoothness of the solution}

We now aim to prove the uniqueness. We divide this step into two sub-steps.

\vskip 0.2 in
{\bf Step 3.1--The weak solution is a strong solution}

 Let $\varrho$ be such a weak solution with 
\[
\sup_{0\le t\le T}\int_{\R^d} |x|\varrho\,dx<C(T).
\] 
Then, $\bar{K}(x,t):=K*\varrho$ is a smooth function (since $K$ is smooth)
and 
\[
|\bar{K}(0)|\le \left|\int_{\R^d} K(x)\varrho(x)\,dx\right| \le |K(0)|+LC(T).
\]
Moreover, it is easy to see that $\bar{K}(x,t)$ is also Lipschitz with the Lipschitz constant bounded by $L$.

We claim that for a given $\varrho$, the solution to 
\[
\begin{split}
&\partial_tu=-\nabla\cdot(b(x)u+\bar{K}(x,t)u)+\sigma^2\Delta  u,\\
&u|_{t=0}=\varrho_0,
\end{split}
\]
is unique and thus must be $\varrho$. In fact, the existence can be justified by the following SDE as its law is a weak solution
\[
dX=(b(X)+\bar{K}(X,t))\,dt+\sqrt{2}\sigma\,dW,~~X_0\sim \varrho_0\,dx.
\]
For the well-posedness of such SDEs, one can refer to \cite[Chap 2, Theorem 3.5]{mao1997}, and also see a recent work with weaker assumptions \cite{trevisan2016well}. Regarding the uniqueness, one considers the difference of two such solutions $u_i, i=1,2$
\[
\partial_t(u_1-u_2)=-\nabla\cdot([b(x)+\bar{K}(x,t)](u_1-u_2))+\sigma^2\Delta  (u_1-u_2).
\]
We then multiply $h_{\e}(u_1-u_2):=h((u_1-u_2)/\e)$ on both sides
and take integral. Here, $h(\cdot)$ is an odd function that increases monotonely
from $-1$ to $1$ on $[-1, 1]$. It is $1$
on $[1,\infty)$. Hence, $h(\cdot/\epsilon)$ is some approximation 
for the sign function.

Then,
\[
\frac{d}{dt}\int_{\R^d} H_{\e}(u_1-u_2)\,dx
\le \int_{\R^d} h'\left(\frac{u_1-u_2}{\e}\right)\frac{u_1-u_2}{\e}(b(x)+\bar{K}(x,t))\cdot \nabla(u_1-u_2)\,dx,
\]
where $H_{\e}(u)=\int_0^u h_{\e}(s)\,ds$. The right hand side goes to zero when $\epsilon\to 0$, because $h'(\frac{u_1-u_2}{\e})\frac{u_1-u_2}{\e}$ is bounded and nonzero only on $|u_1-u_2|\le \e$.
Also, $H_{\e}(u_1-u_2)\to |u_1-u_2|$ as $\e\to 0$. Hence, the claim is shown and thus
\[
u=\varrho.
\]
By the theory of the {\it linear} PDEs, $u=\varrho$ is in fact a strong solution and smooth. For the general theory of linear parabolic equations, one may refer to \cite{friedman2008partial}. 

\vskip 0.2 in
{\bf Step 3.2--The uniqueness of the nonlinear Fokker-Planck equation}

For the uniqueness of the nonlinear Fokker-Planck equation, we cannot use the technique in Step 3.1 as we show uniqueness for the linear PDE, as the term $ K*\varrho$ involves the solution $\varrho$ itself. Also, the classical Dobrushin's estimate \cite{dobrushin1979vlasov, golse2016dynamics} cannot be used because the flow map is not well-defined before we show the uniqueness of $\varrho$.

Instead, we use the interacting particle system for mean-field limit and show that any weak solution is close to the one marginal distribution of the $N$-particle system. This then will result in the uniqueness.

Fix {\it any} weak solution of the nonlinear Fokker-Planck equation. Consider the following SDEs
\begin{gather}\label{eq:appmeanfieldSDE}
dX^i=b(X^i)\,dt+(K*\varrho)(X^i)\,dt+\sqrt{2}\sigma\,dW^i,~i=1,\cdots, N. 
\end{gather}
According to the argument in Step 3.1, the law of each $X^i$
is exactly the weak solution $\varrho$ used to convolve with $K$. Moreover, these $X^i$'s are independent.

Now, consider the interacting praticle system
\begin{gather}\label{eq:appNpart}
dY^i=b(Y^i)\,dt+\frac{1}{N-1}\sum_{j: j\neq i}K(Y^i-Y^j)\,dt
+\sqrt{2}\sigma\,dW^i,~~i=1,\cdots, N.
\end{gather}
The next step is to use the technique in the proof of \cite[Theorem 3.1]{cattiaux2008}.
We compute for fixed $i$,
\begin{multline}
\frac{1}{2}\frac{d}{dt}\E|X^i-Y^i|^2
=\E(X^i-Y^i)\cdot(b(X^i)-b(Y^i)) \\
+\E(X^i-Y^i)\cdot\left(\bar{K}(X^i,t)-\frac{1}{N-1}\sum_{j: j\neq i}K(Y^i-Y^j)\right).
\end{multline}
The first term is controlled by $\beta\E|X^i-Y^i|^2$.
The second term is split as
\[
\begin{split}
&\E(X_i-Y_i)\cdot\left(\bar{K}(X_i,t)-\frac{1}{N-1}\sum_{j: j\neq i}K(Y_i-Y_j)\right)\\
&=\E(X_i-Y_i)\cdot(\bar{K}(X_i,t)-\frac{1}{N-1}\sum_{j: j\neq i}K(X_i-X_j))\\
&+\E(X_i-Y_i)\cdot(\frac{1}{N-1}\sum_{j: j\neq i}K(X_i-X_j)-\frac{1}{N-1}\sum_{j: j\neq i}K(Y_i-Y_j))
=:D_1+D_2.
\end{split}
\]
The term $D_2$ is easily controlled by $2L\E|X_i-Y_i|^2$ by the exchangeability.
For $D_1$, one can control it as
\[
D_1\le \sqrt{\E|X_i-Y_i|^2}\sqrt{\E\left|\bar{K}(X_i,t)-\frac{1}{N-1}\sum_{j: j\neq i}K(X_i-X_j)\right|^2}.
\]
However,
\[
\begin{split}
&\E\left|\bar{K}(X_i,t)-\frac{1}{N-1}\sum_{j: j\neq i}K(X_i-X_j) \right|^2\\
&=\frac{1}{(N-1)^2}\sum_{j,k: j\neq i, k\neq i}\E (\bar{K}(X_i,t)-K(X_i-X_j))(\bar{K}(X_i,t)-K(X_i-X_k)).
\end{split}
\]
By independence, the terms for $j\neq k$ are zero. Hence, only $N-1$ terms will survive. This means 
\[
D_1\le \sqrt{\E|X_i-Y_i|^2}\frac{C_1(T, \varrho)}{\sqrt{N-1}}.
\]
Moreover, $C_1(T, \varrho)$ will have an upper bound that is independent of $T$
if  Assumption \ref{ass:kernelstrong} holds.

By Gr\"onwall's inequality,
\[
\sqrt{\E|X_i-Y_i|^2}\le C(T, \varrho)\frac{1}{\sqrt{N-1}}.
\]
Hence, for any two weak solutions $\varrho_1, \varrho_2$, we have
\[
\sup_{0\le t\le T}W_2(\varrho_1, \varrho_2)\le  [C(T, \varrho_1)+C(T,\varrho_2)]\frac{1}{\sqrt{N-1}}.
\]
Taking $N\to\infty$ yields the uniqueness of the solutions to the nonlinear Fokker-Planck equation.

\vskip 0.1 in

{\bf Step 4--Strong confinement}

Under  Assumption \ref{ass:kernelstrong}, one in fact has
\[
I_1+I_2+I_3
\le q(-r+2L+\delta)\int_{\R^d}|x|^q\varrho\,dx+C(\delta).
\]
The assertions about moments have then been proved with application of Gr\"onwall's inequality.

Under this condition, the estimate of $D_1$ term in Step 3
can also be independent of $T$, because of this uniform moment control. Hence, the mean field limit can be uniform in $T$.

Lastly, to show the convergence of $\varrho$ as $t\to\infty$, we consider two different initial data $\varrho_{j,0}$ where $j=1,2$. Then, one can consider \eqref{eq:appNpart}
with these two initial data. 
Pick the coupling between $Y_1^i(0)$ and $Y_2^i(0)$ (the data for different $i$'s are independent) such that
\[
\E|Y_{1}^i(0)-Y_2^i(0)|^{\sq}\le W_{\sq}^{\sq}(\varrho_{1,0}, \varrho_{2,0})+\epsilon,~\forall i=1,\cdots, N.
\]
Then, by similar computation,
\[
\frac{d}{dt}\E|Y_{1}^i(t)-Y_2^i(t)|^{\sq}\le \sq (-r+2L)\E|Y_{1}^i(t)-Y_2^i(t)|^{\sq}.
\]
Fixing $t>0$ and taking $N\to\infty$, $\mathscr{L}(Y_j^i(t))\to \varrho_j(t), j=1,2$. Hence, the evolutional nonlinear semigroup for the nonlinear Fokker-Planck equation is a contraction
\[
W_{\sq}(\varrho_1(t), \varrho_2(t))\le W_{\sq}(\varrho_{1,0}, \varrho_{2,0})e^{-(r-2L) t}.
\]
 Thus, the last claim follows.

\section{Proof of Lemma \ref{lmm:fundamentalmoments}}\label{app:fundmoment}

Since $\sigma>0$, without loss of generality, we will assume
\[
\sigma\equiv 1.
\]
We first fix $s\ge 0$. 
Consider the trajectory determined by
\begin{gather}
\partial_t Z(t; y, s)=b(Z, t),~~Z(s; y, s)=y.
\end{gather}

Then, one has
\[
\frac{1}{2}\frac{d}{dt}|Z|^2\le \beta_1|Z|^2+C|Z|
\]
as $b(0,t)$ is bounded.
Hence,
\[
\frac{d}{dt}|Z|\le \beta_1|Z|+C.
\]
This means 
\begin{gather}\label{eq:controlZ}
|Z|\le |y|e^{\beta_1(t-s)}+C\int_s^t e^{\beta_1(t-s)}\,ds.
\end{gather}

Moreover, \eqref{eq:b1cond1} implies that
\[
v\cdot\nabla b_1(x,t)\cdot v\le \beta_1|v|^2,~~\forall v, x\in \R^d, t\ge 0.
\]
Consequently
\begin{gather}\label{eq:Zjaboci}
|\nabla_y Z|\le \sqrt{d} e^{\beta_1 (t-s)},
\end{gather}
uniform in $y$, where $|A|:=\sqrt{\sum_{ij}A_{ij}^2}$ is the matrix Frobenius norm.

Assume without loss of generality $|x|\ge  |y|$. Clearly,
\[
|b_1(x, t)-b_1(y,t)|
\le |x-y|\left|\int_0^1\nabla b_1(x\theta+y(1-\theta),t)d\theta\right|.
\]
 Due to the assumption of polynomial growth of derivatives of $b_1$, 
\[
|\nabla b_1(xz+y(1-z),t)|\le C(1+|x\theta+y(1-\theta)|^q).
\]
If $|y|\le \frac{1}{2}|x|$, then $|x\theta+y(1-\theta)|\le \frac{3}{2}|x|\le 3|x-y|$.
Otherwise, we bound this by a polynomial of $|y|$ directly. Hence,
\begin{gather}\label{eq:differenceb1}
|b_1(x, t)-b_1(y, t)|\le \min(P_1(|x|), P_1(|y|))|x-y|+P_2(|x-y|)|x-y|
\end{gather}
for some polynomials $P_1, P_2$.

We denote
\begin{gather}
\Phi_0(x,t; y,s):=\frac{1}{(2\pi (t-s))^{d/2}}\exp\left(-\frac{|x-Z(t;y,s)|^2}{2(t-s)}\right).
\end{gather}

Below, we establish an important lemma indicating that $\Phi_0$ is the main term of $\Phi$, and Lemma \ref{lmm:fundamentalmoments} will follow easily.
\begin{lemma}\label{lmm:funddecomp}
It holds that
\begin{gather}
\Phi(x, t; y, s)=\Phi_0(x,t; y,s)+u(x,t; y,s),
\end{gather}
where $u$ satisfies
\begin{gather}
\int_{\R^d} (1+|x|^q)|\nabla_y u|\,dx\le h(t-s)P(|y|),
\end{gather}
for some polynomial $P(\cdot)$, some nondecreasing function $h(\cdot)$ defined on $[0,\infty)$. 

Moreover, if $\beta_1<0$, $h(t-s)$ can be taken as
\begin{gather}
h(t-s)=Ce^{-\delta_1(t-s)}
\end{gather}
for some $\delta_1>0$.
\end{lemma}

\begin{proof}
It is not hard to verify
\begin{gather}
\partial_t\Phi_0+\nabla_x\cdot(b_1(x,t)\Phi_0)
-\Delta_x\Phi_0=\nabla_x\cdot([b_1(x,t)-b(Z,t)]\Phi_0).
\end{gather}
Hence, letting $u=\Phi-\Phi_0$, one finds
\begin{gather}
\begin{split}
&\partial_tu+\nabla_x\cdot(b_1(x,t)u)
-\Delta_x u=-\nabla_x\cdot([b_1(x,t)-b(Z,t)]\Phi_0),\\
&u|_{t=s}=0.
\end{split}
\end{gather}
Letting
\[
v:=\partial_{y_i}u,
\]
one has
\begin{gather}
\begin{split}
&\partial_tv+\nabla_x\cdot(b_1(x,t)v)
-\Delta_x v=R,\\
&u|_{t=s}=0,
\end{split}
\end{gather}
where
\[
R:=\nabla_x\cdot b_1(x,t) \nabla\Phi_0\cdot \partial_{y_i}Z
+\partial_{y_i}Z\cdot\nabla b_1(x,t)\cdot\nabla\Phi_0
+(b_1(x,t)-b_1(Z,t))\cdot\nabla^2\Phi_0\cdot\partial_{y_i}Z.
\]
Writing $\nabla_x\cdot b_1(x,t)=
[\nabla_x\cdot b_1(x,t)-\nabla\cdot b_1(Z,t)]+\nabla\cdot b_1(Z,t)$, it is not hard to see (using also \eqref{eq:Zjaboci} and \eqref{eq:differenceb1})
\[
|R|\le P(|Z|)\frac{1}{(t-s)^{(d+1)/2}}\exp\left(-\frac{\gamma |x-Z|^2}{2(t-s)}\right)e^{\beta_1(t-s)}
\]
for some polynomial $P$ and $\gamma \in (0, 1)$.

We then find
\[
v=\int_s^t S_{\lambda,t}R\,d\lambda.
\]

Below, we use $h_i(\cdot)$ to denote some nondecreasing functions defined on $[0,\infty)$.
By Lemma \ref{lmm:momentevol}, one has
\[
\int_{\R^d} (1+|x|^q)|v|\,dx
\le h_1(t-s)\int_s^t \int_{\R^d}(1+|x|^q)|R(x,\lambda)|\,dxd\lambda.
\]

Clearly,
\[
\int_{\R^d} (1+|x|^q)\frac{1}{(t-s)^{(d+1)/2}}\exp\left(-\frac{\delta |x-Z|^2}{2(t-s)}\right)\,dx
\le C\frac{1+(t-s)^{q/2}}{\sqrt{t-s}}(1+|Z|^q).
\]
Moreover, by the stability of trajectory of $Z$ \eqref{eq:controlZ},
$P(|Z|) \le h_2(t-s)P(|y|)$.
Hence,
\[
\int_{\R^d} (1+|x|^q)|v|\,dx
\le h_3(t-s)P(|y|)\int_s^t \frac{1}{\sqrt{\lambda-s}} d\lambda. 
\]

If $\beta_1<0$, we consider $t\ge s+1$ and
\begin{gather}\label{eq:vsplit}
v=\int_s^t S_{\lambda,t}R\,d\lambda
=S_{(t+s)/2,t}\int_{s}^{\frac{t+s}{2}}S_{\lambda,(t+s)/2}R\,d\lambda
+\int_{(t+s)/2}^tS_{\lambda,t}R\,d\lambda.
\end{gather}
The second term is like
\[
\begin{split}
\int_{\R^d} (1+|x|^q)|v|\,dx
&\le C\int_{(s+t)/2}^t \int_{\R^d}(1+|x|^q)|R(x,\lambda)|\,dxd\lambda\\
& \le C\int_{(s+t)/2}^t e^{\beta_1(\lambda-s)}P(|Z|)
\int_{\R^d}\frac{1+|x|^q}{(t-s)^{(d+1)/2}}\exp\left(-\frac{\delta |x-Z|^2}{2(t-s)}\right)\,dxd\lambda.
\end{split}
\]
This is easily controlled by $P(|y|)e^{-\delta_1(t-s)}$
for some polynomial $P$ and $\delta_1>0$ (recall \eqref{eq:controlZ}).

For the first term in \eqref{eq:vsplit}, we note
$\int_{s}^{\frac{t+s}{2}}S_{\lambda,(t+s)/2}R\,d\lambda\in L^1$,
and
\[
\int_{\R^d}\int_{s}^{\frac{t+s}{2}}S_{\lambda,(t+s)/2}R\,d\lambda\,dx=0
\]
since $\int R(x,\lambda)\,dx=0$ for all $\lambda$. Hence, statement (ii) in Lemma \ref{lmm:momentevol} implies that
\begin{multline*}
\int_{\R^d}(1+|x|^q)\left|S_{(t+s)/2,t}\int_{s}^{\frac{t+s}{2}}S_{\lambda,(t+s)/2}R\,d\lambda\right|\,dx \\
\le e^{-\delta(t-s)/2}P\left(M_{q_1}\left(\Big|\int_{s}^{\frac{t+s}{2}}S_{\lambda,(t+s)/2}R\,d\lambda\Big|\right)\right).
\end{multline*}
For the inside
\[
M_{q_1}\left(\left|\int_{s}^{\frac{t+s}{2}}S_{\lambda,(t+s)/2}R\,d\lambda \right|\right)
\le C\int_{s}^{(t+s)/2}\int_{\R^d}(1+|x|^{q_1})|R|\,dxd\lambda,
\]
where $C$ is independent of time as $\beta_1<0$. As has been proved, 
the integral here is controlled by products of polynomials of $|y|$, $|t-s|$.
Hence, the first term is also controlled similarly.
\end{proof}

As Lemma \ref{lmm:funddecomp} is proved, 
Lemma \ref{lmm:fundamentalmoments} is very straightforward, since
$|\nabla_y Z|\le Ce^{\beta_1(t-s)}$.

\bibliographystyle{plain}
\bibliography{sdealg}

\end{document}